\documentclass[a4paper,11pt]{article}

\usepackage[utf8]{inputenc}
\usepackage{lmodern} 

\usepackage[margin=3cm]{geometry}

\usepackage{amsfonts,amssymb,amsmath,amsthm}
\usepackage{stmaryrd} 
\usepackage{latexsym}
\usepackage{xcolor}
\usepackage{bbm}
\usepackage{soul} 
\usepackage{enumerate}
\usepackage{bbm}
\usepackage{nicefrac}
\usepackage{hyperref}

\newtheorem{thm}{Theorem}[section]
\newtheorem{lem}[thm]{Lemma}

\newtheorem{prop}[thm]{Proposition}

\theoremstyle{remark}
\newtheorem{rem}{Remark}
\newtheorem{exa}{Example}
\newtheorem{count-exa}{Counter-example}

\DeclareFontEncoding{FMS}{}{}
\DeclareFontSubstitution{FMS}{futm}{m}{n}
\DeclareFontEncoding{FMX}{}{}
\DeclareFontSubstitution{FMX}{futm}{m}{n}
\DeclareSymbolFont{fouriersymbols}{FMS}{futm}{m}{n}
\DeclareSymbolFont{fourierlargesymbols}{FMX}{futm}{m}{n}
\DeclareMathDelimiter{\VERT}{\mathord}{fouriersymbols}{152}{fourierlargesymbols}{147}

\def\E{\mathbb{E}}
\def\P{\mathbb{P}}
\def\R{\mathbb{R}}
\def\Q{\mathbb{Q}}
\def\1{\mathbbm{1}}
\def\N{\mathbb{N}}

\def\Z{\mathbb{Z}}

\newcommand{\cW}{\mathcal{W}}

\begin{document}

\title{Uniform Wasserstein convergence of penalized Markov processes}
\author{Nicolas Champagnat$^{1}$, \'Edouard Strickler$^{1}$, Denis Villemonais$^{1}$}

\footnotetext[1]{Universit\'e de Lorraine, CNRS, Inria, IECL, F-54000 Nancy, France \\
  E-mail: Nicolas.Champagnat@inria.fr, edouard.strickler@univ-lorraine.fr, Denis.Villemonais@univ-lorraine.fr }

\maketitle

\begin{abstract}
  {For general penalized Markov processes with soft killing, we propose a simple criterion ensuring uniform convergence
    of conditional distributions in Wasserstein distance to a unique quasi-stationary distribution. We give several examples of
    application where our criterion can be checked, including Bernoulli convolutions and piecewise deterministic Markov processes of
    the form of switched dynamical systems, for which convergence in total variation is not possible.}
\end{abstract}




\section{Introduction}
\label{sec:introduction}

Consider the Bernoulli convolution
\begin{equation}
  \label{eq:Bernoulli}
  X_{n+1} = \frac{1}{2} X_n + \theta_{n+1}
\end{equation}
with $X_0=x\in[-2,2]$ and $(\theta_n)_{n\geq 1}$ i.i.d.\ with $\mathbb{P}( \theta_1 = 1) = \mathbb{P}( \theta_1 = - 1) = 1/2$. This is a
Markov chain on $E = [ - 2, 2]$. It is well-known that this Markov chain admits as unique invariant distribution $\mu$ the uniform
probability measure on $[-2,2]$. For all $x,y\in[-2,2]$, the trivial coupling between the chains started from $x$ and $y$ defined as
$X_0=x$, $Y_0=y$,
\[
  X_{n+1} = \frac{1}{2} X_n + \theta_{n+1}\quad\text{and}\quad Y_{n+1} = \frac{1}{2} Y_n + \theta_{n+1}
\]
satisfies
\[
  |X_n-Y_n|\leq 2^{-n}|x-y|.
\]
It is then trivial to prove the geometric convergence in $L^1$ Wasserstein distance of the law of $X_n$ to $\mu$.
  
If we now assume that the Markov chain $(X_n)_{n\geq 0}$ is killed at time $n$ with probability $1-p(X_{n-1})$, for some measurable
function $p:[-2,2]\to[0,1]$, proving convergence in Wasserstein distance of the law of $X_n$ conditioned to survive is far less immediate and,
to the best of our knowledge, no general criterion allowing to prove such results is available in the literature. This is the goal of
this article. Namely, we introduce general criteria implying the convergence in Wasserstein distance of conditional
distributions of absorbed Markov processes given non-absorption, to a so-called quasi-stationary distribution.

General criteria are known for convergence of distributions of absorbed Markov processes conditioned to non-absorption in total
variation norm, based on a variety of methods, cf.\
e.g.~\cite{Birkhoff1957,DelMoral2004,ChampagnatVillemonais2016b,BansayeEtAl2020,GuillinEtAl2020,FerreRoussetEtAl2021,BenaimChampagnatEtAl2022,ChampagnatVillemonais2023}.
However, in the case of Bernoulli convolution, we cannot expect to have convergence in total variation distance because, if $X_0$ is
rational (resp.\ irrational), $X_n$ is also rational (resp.\ irrational) for all $n\geq 0$, so that the laws of $X_n$ conditioned to
survival started from a rational and an irrational initial value are singular. Because of this singularity, one can not expect to obtain convergence for the total variation distance. In addition, the methods used in the above references, which rely either on regularising properties of the semi-group or on the existence of small sets (which allow to build couplings in term of total variation distance), are not applicable in our context.

Far fewer references deal with convergence in Wasserstein distance of conditional distributions.
In~\cite{CloezThai2016,Villemonais2020,JournelMonmarche2022}, the authors prove the uniform convergence in Wasserstein distance of
particle systems approximating conditional distributions and deduce similar properties for the killed process. However, all these
references assume strong assumptions on the killing probability (or rate in continuous time) which amounts to assume that coupling of
unkilled processes is faster that killing (see Section~\ref{sec:coupling-faster-killing} below).  In~\cite{Ocafrain2021}, the author proves convergence in Wasserstein distance
of conditional distributions of diffusions absorbed at the boundary of a domain under assumptions ensuring that the diffusion
satisfies a Poincar\'e inequality. In~\cite{delMoral2021}, the authors study quantum (Shr\"odinger) harmonic operators with Gaussian
initial distributions and with quadratic potential energy, for which they prove convergence in Wasserstein distance of the
Feynman-Kac semigroup using that Gaussian distributions are preserved by the system. Let us also cite~\cite{Ocafrain2020a}, where
convergence in Wasserstein distance of conditional distributions of normalized drifted Brownian motions conditioned to non-absorption
at 0 is proved, based on explicit formulas for the distribution of the process.
Note that most of these references consider $L^1$ Wasserstein distances, except~\cite{delMoral2021} which considers $L^2$ Wasserstein
distances. None are applied to processes for
which convergence in total variation is not expected. 
One of the main difficulties is that estimates in Wasserstein distances based on coupling do not extend easily to
conditional distributions. This is related to the fact that extending Wasserstein distances to measures with different masses is a
hard problem~\cite{FigalliGigli2010,Peyre2013,PiccoliRossi2014,PiccoliRossi2016,KondratyevEtAl2016}.

Assuming that the killing probability (or rate in continuous time) is Lipschitz with respect to a bounded distance, the criterion we
propose implies convergence in $L^1$ Wasserstein distance of conditional distributions to a quasi-stationary distribution.
Convergence is exponential and uniform with respect to the initial condition, with explicit (although non-optimal) convergence rate.
We also obtain asymptotic estimates on the survival probability, ergodicity in Wasserstein distance of the $Q$-process (the process
conditioned to never be absorbed), and quasi-ergodic estimates in Wasserstein distance. After showing that our criterion improves
known results which assume that coupling is faster than killing, we consider processes extending the Bernoulli
convolution~\eqref{eq:Bernoulli} both in continuous and discrete time, under the only assumption that the killing probability is
Lipschitz and uniformly bounded away from 0 (with counter-examples when it is not-Lipschitz or not bounded away from 0). We conclude
with an application to Piecewise Deterministic Markov Processes (PDMP) taking the form of switched ordinary differential equations.

In Section~\ref{sec:results}, we state our main criterion and give the main general results on convergence to a quasi-stationary
distribution, asymptotic behavior of the survival probability, ergodicity of the $Q$-process and quasi-ergodic behavior. The rest of
Section~\ref{sec:criterion} is dedicated to the proof of these results. We then propose in Section~\ref{sec:coupling-faster-killing}
general criteria relating the coupling rate of the non-penalized process and the killing rate, covering cases where coupling is
faster than killing. Applications to processes with almost-sure contraction of unkilled trajectories are given in
Section~\ref{sec:almost-sure-contraction}. The case of iterated random contracting functions extending the Benoulli convolution is
then studied in Section~\ref{sec:random-composition}, without assuming uniform contraction. We conclude with applications to switched
dynamical systems in Section~\ref{sec:switched}.

\section{A general criterion for uniform exponential convergence of conditional distributions in Wasserstein distance}
\label{sec:criterion}

\subsection{Results}
\label{sec:results}

Let $(E,d)$ be a Polish space. Recall the definition of the $L^1$-Wasserstein distance $\mathcal{W}_{d}$ on the set of probability
measures on $E$: for all probability measures $\mu$ and $\nu$ on $E$ as 
\[
  \mathcal{W}_{d}(\mu,\nu)=\inf_{\pi\in\Pi(\mu,\nu)}\iint_{E^2}d(x,y)\,\pi(dx,dy),
\]
where $\Pi(\mu,\nu)$ is the set of probability measures on $E^2$ with first marginal $\mu$ and second marginal $\nu$. The infimum in
the definition above is actually attained on $\Pi(\mu,\nu)$~\cite[Thm.\,4.1]{Villani2009}. All such minimizing coupling is called an
\emph{optimal coupling}. Recall also the Kantorovich-Rubinstein duality relation~\cite[Thm.\,5.10]{Villani2009}: for all probability
measures $\mu$ and $\nu$ on $E$,
\[
  \mathcal{W}_{d}(\mu,\nu)=\sup_{\phi\in\text{Lip}_1(d)}\left|\int_E \phi(x)\mu(dx)-\int_E\phi(y)\nu(dy)\right|,
\]
where $\text{Lip}_1(d)$ is the set of $d$-Lipschitz functions $\phi$ on $E$ with Lipschitz norm $\|\phi\|_{\text{Lip}(d)}\leq
1$, where
\[
  \|\phi\|_{\text{Lip}(d)}=\sup_{x\neq y\in E}\frac{|\phi(x)-\phi(y)|}{d(x,y)}.
\]
We also denote by $\text{Lip}(d)$ the set of $d$-Lipschitz functions from $E$ to $\mathbb{R}$. Note that, if $d$ is bounded,
$\text{Lip}(d)\subset L^\infty(E)$.

Let $\left(\Omega,(\mathcal{F}_{t})_{t\in I},(\P_x)_{x\in E},(X_t)_{t\in I}\right)$ be a Markov process evolving in $E$, where the
time space is either $I=[0,+\infty)$ or $I=\Z_+$. Since our criteria make use of coupling between unkilled processes, it is more
convenient to work in the formalism of penalized Markov semigroups than killed Markov processes. Let $Z=\{Z_{t};t\in I\}$ be a
collection of multiplicative nonnegative random variables defined either as
\begin{equation}
  \label{eq:penaliz-discret}
  Z_t=Z_t(X)=p(X_0)p(X_1)\ldots p(X_{t-1})\ \text{for\ } I=\Z_+,
\end{equation}
or
\begin{equation}
  \label{eq:penaliz-contin}
  Z_{t}=Z_t(X)=e^{-\int_0^t \rho(X_s) ds}\ \text{for\ } I=\R_+,
\end{equation}
where $p$ (resp.\ $\rho$) is a measurable function from $E$ to $[0,1]$ (resp.\ $\R_+$). For notational convenience, in the
discrete-time case, we write $p(x)=\exp(-\rho(x))$. Throughout all this work, we will make the following assumption on the distance
$d$ and the function
$\rho$, both in the discrete-time and continuous-time cases, which implies in particular that
$\text{osc}(\rho):=\sup_{x\in E}\rho(x)-\inf_{x\in E}\rho(x)$ is finite.

\paragraph{Standing assumption}
The distance $d$ is bounded on $E\times E$ by a constant $\bar{d}$ and the function $\rho:E\to\mathbb{R}_+$ is $d$-Lipschitz.

\bigskip


We define the penalized (non-conservative) semigroup, for all bounded measurable $f$ on $E$ and all $x\in E$, as
\[
P_t f(x)=\E_{x}[Z_{t} f(X_t)],\quad\forall t\in I.
\]
We extend this definition as usual to any probability measure $\mu$ on $E$ as
\[
\mu P_t f=\int_E P_t f(x)\,\mu(dx)
\]
In particular, $\delta_x P_{t} f= P_t f(x)$. The quantity $p(x)$ (resp.\ $\rho(x)$ in continuous time) can be interpreted as a
killing probability (resp.\ rate) in position $x$, so that the probability measure $\frac{\delta_x P_{t}}{\delta_x P_t\mathbbm{1}}$
is the conditional distribution of $X_t$ given $X_0=x$ and that the process $X$ is not killed before time $t$.

\begin{rem}
  Note that the assumption that $Z_t\leq 1$ could be relaxed following the ideas of~\cite{ChampagnatVillemonais2020}. In particular,
  if $Z_t$ is bounded by the constant $e^{\lambda t}$ for all $t\geq 0$, then defining
  $\delta_x P_t f=e^{-\lambda t}\E_{x}[Z_{t} f(X_t)]$ would allow to recover the present framework.
\end{rem}

\begin{rem}
The assumption that $(X_t,Z_t)_{t\geq 0}$ is a time-homogeneous penalized Markov process could be also relaxed easily and most of the
results below can be extended to the time-inhomogeneous case, in the same spirit
as~\cite{ChampagnatVillemonais2018b,BansayeEtAl2020,Ocafrain2020}.  
\end{rem}

Recall the definition of a quasi-stationary distribution for the penalized process $X$: a probability measure $\nu$ on $E$ is called a
\emph{quasi-stationary distribution} if
\[
  \frac{\nu P_t}{\nu P_t\mathbbm{1}}=\nu,\quad\forall t\geq 0.
\]
It is well-known (cf.\ e.g.~\cite{ColletMartinezEtAl2013,MeleardVillemonais2012,DoornPollett2013}) that to any quasi-stationary distribution $\nu$ is associated a
nonnegative number $\lambda_0$, called the \emph{absorption rate} of the quasi-stationary distribution $\nu$, such that
\[
  \nu P_t=e^{-\lambda_0 t}\nu,\quad \forall t\geq 0.
\]

For all $x,y\in E$, we call \emph{coupling measure between $\mathbb{P}_x$ and $\mathbb{P}_y$} a probability measure $\P_{(x,y)}$ on a
probability space where is defined $(X_t,Y_t)_{t\in I}$ such that $(X_t)_{t\in I}$ (resp.\ $(Y_t)_{t\in I}$) has the same
distribution as $(X_t)_{t\in I}$ under $\mathbb{P}_x$ (resp.\ under $\mathbb{P}_y$). We say that a coupling measure $\P_{(x,y)}$ is
\emph{Markovian} if the coupled process $(X_t, Y_t)_{t \in I}$ is Markov  with respect to its natural filtration.
We also define, for all $t\in I$, $Z_t^X=Z_t(X)$, $Z_t^Y=Z_t(Y)$,
\[
  G_{t}^X=\frac{Z^X_{t}}{\E_{x} Z_{t}}\quad\text{and}\quad
  G_{t}^Y=\frac{Z^Y_{t}}{\E_{y} Z_{t}}.
\]
In particular, $\E_{(x,y)}Z^X_t=\E_x Z_t$.

We introduce our main assumption
\begin{description}
\item[(A)]  (exponential penalized coupling) There exist $C_A, \gamma_A > 0$ such that, for all $x,y\in E$,
  there exists a Markovian coupling $\P_{(x,y)}$ between $\mathbb{P}_x$ and $\mathbb{P}_y$ such that for all $t\in I$,
  \begin{equation*}
    \E_{(x,y)}\left[G^X_{t}d(X_t,Y_t)\right]\leq 
    C_A e^{-\gamma_A t} d(x,y).
  \end{equation*} 
\end{description}

This assumption can be interpreted as a coupling estimate between $X_t$ and $Y_t$  with penalization depending only on
  the paths of $X$. This
only assumption is sufficient to ensure convergence in Wasserstein distance to a unique quasi-stationary distribution.


\begin{thm}
  \label{thm:main-compact}
  Under Assumption~(A), there exist constants $C_1>0$ and $\alpha>0$ such that, for all probability measures $\mu$ and
  $\nu$ on $E$ and all $t\in I$,
  \begin{equation}
    \label{eq:main-compact}
    \mathcal{W}_{d}\left(\frac{\mu P_t}{\mu P_t\mathbbm{1}}, \frac{\nu P_t}{\nu P_t\mathbbm{1}}\right)\leq C_1 e^{-\alpha
      t}\mathcal{W}_{d}(\mu,\nu).    
  \end{equation}
  In addition, the penalized semi-group $(P_t)_{t\in I}$ admits a unique quasi-stationary distribution $\nu_\text{QS}$ such that
  \begin{equation}
    \label{eq:hyp-eta}
    \mathcal{W}_{d}\left(\frac{\mu P_t}{\mu P_t\mathbbm{1}},\nu_\text{QS}\right)\leq C_1 e^{-\alpha
      t}\mathcal{W}_{d}(\mu,\nu_\text{QS}).    
  \end{equation}
\end{thm}

\begin{rem}
  \label{rem:explicite-1}
  The constants $C_1$ and $\alpha$ in~\eqref{eq:main-compact} and~\eqref{eq:hyp-eta} can be expressed explicitely in terms of
  $\bar{d}$, $\|\rho\|_{\text{Lip}(d)}$, $C_A$ and $\gamma_A$, as can be checked from the proof. However, these constants are not
  optimal and quite complicated to obtain. We will see in Theorem~\ref{thm:main-under-ABC} alternative assumptions allowing to obtain
  better estimates on $C_1$ and $\alpha$.
\end{rem}

\begin{rem}
  \label{rem:non-borne} The fact that the convergence in~\eqref{eq:main-compact} is uniform with respect to the initial distributions
  (i.e.\ that $C_1$ does not depend on $\mu$ and $\nu$) is strongly related to the fact that the distance $d$ is assumed to be
  bounded. In the case of an unbounded distance, convergence of non-penalized processes is known to be related to Foster-Lyapunov
  criteria (see e.g.~\cite{MeynTweedie2009} for the total variation distance or~\cite{HMS11} for the Wasserstein distance). However, the
  question of convergence in Wasserstein distance of penalized Markov processes is completely open in this setting and we leave this study to further
  work.
\end{rem}

The proof of Theorem~\ref{thm:main-compact} is divided into two parts. First, in Section~\ref{sec:proofunderABC}, we prove the result
of the theorem under a set of alternative assumptions. Then, in Section~\ref{sec:main-proof-2}, we show that our assumption (A) is
equivalent to these alternative assumptions. In particular, we shall prove that (A) implies the following key property, which may be
interpreted as a global Harnack inequality for survival probabilities:
 \begin{description}
 \item[(H)]  (global Harnack inequality) There exist $C_H>0$ such that, for all $x,y\in E$ and $t\in I$,
   \begin{equation*}
     \E_x(Z_t) \leq C_H \E_y(Z_t).
   \end{equation*}
\end{description}

\begin{prop}
  \label{prop:Aexp-implies-H}
  Under Assumption (A), Condition~(H) is satisfied.
\end{prop}
{This proposition is proved in Section~\ref{sec:proofofABC}.} Before proceeding with the proof {of Theorem~\ref{thm:main-compact}}, we give other consequences of Assumption (A) on the
quasi-stationary behavior of the process. First, the next result, proved in Section~\ref{sec:proof-eta}, deals with the asymptotic
behaviour of the survival probability.

\begin{thm}
  \label{thm:eta-compact}
  Assume (A).
  Then, there exists a $d$-Lipschitz function $\eta:E\rightarrow(0,+\infty)$ such that
  \[
  \eta(x)=\lim_{t\rightarrow+\infty} e^{\lambda_0 t}\E_x Z_t,
  \]
  where $\lambda_0$ is the absorption rate of the quasi-stationary distribution $\nu_{QS}$ and the convergence holds
  exponentially fast for the uniform norm. The function $\eta$ also satisfies that $\nu_{QS}(\eta)=1$,
  $P_t\eta(x)=e^{-\lambda_0 t}\eta(x)$ for all $x\in E$ and $t\in I$ and is uniformly bounded away from 0 on $E$.
\end{thm}
The last result also has consequences on the $Q$-process, informally defined as the process $(X_t)_{t\geq 0}$ penalized up to
$+\infty$. The next result is proved in Section~\ref{sec:proof-Q-proc}.

\begin{thm}
  \label{thm:Q-process-compact}
  Under Assumptions (A), we have the following properties.
  \begin{description}
  \item[\textmd{(i) Existence of the $Q$-process.}] There exists a family $(\Q_x)_{x\in E}$ of probability measures on $\Omega$
    such that, for some constant $C$,
    \begin{equation}
      \label{eq:Q-process-compact}
      \left|\frac{\E_x(\mathbbm{1}_A Z_t)}{\E_x Z_t}-\Q_x(A)\right|\leq C e^{-\alpha (t-s)}
    \end{equation}
    for all ${\cal F}_s$-measurable set $A$ for all $s\leq t\in I$, where the constant $\alpha$ is the same as in
    Theorem~\ref{thm:main-compact}. The process $(\Omega,({\cal F}_t)_{t\in I},(X_t)_{t\in I},(\Q_x)_{x\in E})$ is an $E$-valued
    homogeneous Markov process. In addition, if $(X_t)_{t\in I}$ is a strong Markov process under $(\P_x)_{x\in E}$, then so is
    $(X_t)_{t\in I}$ under $(\Q_x)_{x\in E}$.
  \item[\textmd{(ii) Transition semigroup.}] The semigroup $(\widetilde{P}_t)_{t\in I}$ of the Markov process $(X_t)_{t\geq 0}$ under
    $(\Q_x)_{x\in E}$ is given by
    \begin{equation}
      \label{eq:semi-group-Q}
      \widetilde{P}_t\varphi(x)=\frac{e^{\lambda_0 t}}{\eta(x)}P_t(\eta\varphi)(x)
    \end{equation}
    for all bounded measurable $\varphi$ and $t\in I$.
  \item[\textmd{(iii) Exponential ergodicity in Wasserstein distance.}] The probability measure $\nu_Q$ on $E$ defined by
    \[
    \nu_Q(dx) =\eta(x)\nu_{QS}(dx).
    \]
    is the unique invariant distribution of $(X_t)_{t\in I}$ under $(\Q_x)_{x\in E}$. In addition, for any initial
    distributions $\mu$ and $\nu$ on $E$,
    \begin{equation}
      \label{eq:Q-proc-compact}
      \mathcal{W}_{d}(\Q_{\mu}(X_t\in\cdot),\Q_{\nu}(X_t\in\cdot))\leq \frac{C_0(1\vee(\kappa\bar{d}))}{\beta}\,e^{-\alpha t}\,\mathcal{W}_d(\mu,\nu),
    \end{equation}
    where $\Q_\mu=\int_E \Q_x\,\mu(dx)$ and the constants $\kappa$, $C_0$ and $\beta$ are defined in the proof of
    Theorem~\ref{thm:main-compact} (see Lemma~\ref{lem4bis}).
  \end{description}
\end{thm}

Finally, we prove the existence and convergence to a unique quasi-ergodic distribution (for general definition and properties of
quasi-ergodic distribution, we refer to~\cite{BreyerRoberts1999,ChampagnatVillemonais2016}; see also~\cite{Wang2021} where a similar
result as the following one is proved for a certain class of diffusion processes): for all $x\in E$ and $t\in I\setminus\{0\}$, we
define the probability measure
\[
  \mu^x_t=\frac{\E_x\left[Z_t\left(\frac{1}{t}\int_0^t \delta_{X_s}ds\right)\right]}{\E_x Z_t}\text{ if }I=[0,+\infty)\text{ or
  }\mu^x_t=\frac{\E_x\left[Z_t\left(\frac{1}{t}\sum_{s<t}\delta_{X_s}\right)\right]}{\E_x Z_t}\text{ if }I=\Z_+,
\]
which can be interpreted as the mean occupation measure of the process $X$ conditioned to non-absorption.

\begin{thm}
  \label{thm:QED}
  Under Assumption (A), there exists a contant $C>0$ such that, for all $x\in E$ and $t\in I\setminus\{0\}$,
  \[
    \mathcal{W}_d\left(\mu^x_t,\nu_Q\right)\leq \frac{C}{t}.
  \]
\end{thm}
This theorem is proved in Section~\ref{sec:proof-QED}.

\subsection{Proof of Theorem~\ref{thm:main-compact} under  additional assumptions}
\label{sec:proofunderABC}
We introduce the following additional assumptions:
\begin{description}
\item[(B)]  (log-Lipschitz survival probability)
  There exists a constant $C_B>0$ such that the coupling measure $\P_{(x,y)}$ of~(A)
  also satisfies, for all $x,y\in E$ and $t\in I$,
  \[
    \E_{(x,y)} |Z_{t}^X-Z_{t}^Y|\leq C_B d(x,y)\E_{x} Z_{t}.
  \]
\item[(C)]    (compatibility of penalization and coupling) There exists a constant $C_C>1$ such that the coupling measure $\P_{(x,y)}$ of~(A)
  also satisfies that, for all $t\in I$ and $x,y\in E$,
  \[
    \E_{(x,y)} \left[G^X_{t}\wedge G^Y_{t}\right]\geq \frac{1}{C_C}.
  \]
\end{description}
In this section, we prove 
\begin{thm}
\label{thm:main-under-ABC}
The conclusions of Theorem \ref{thm:main-compact} hold true under assumptions  (A), (B) and (C) with
  \begin{equation}
    \label{eq:alpha-explicite}
    \alpha=\frac{\gamma_A\log(\frac{2C_C}{2C_C-1})}{\log\left(2 C_A(1+C_B\bar d)[1\vee\frac{2 C_B\bar d C_C^2}{C_C-1}]\right)}.
  \end{equation}
\end{thm}

 In Section~\ref{sec:main-proof-2}, we will prove that (A) implies Conditions~(B),and~(C), so
  that Theorem~\ref{thm:main-compact} follows from Theorem~\ref{thm:main-under-ABC}. However, the constants $C_B$ and $C_C$ obtained
  from this proof are not sharp, and we will see in the examples of Section~\ref{sec:almost-sure-contraction} that sharper constants
  can sometimes be obtained directly. In this case, Theorem~\ref{thm:main-under-ABC} provides better bounds than
  Theorem~\ref{thm:main-compact}. This explains why we state Theorem~\ref{thm:main-under-ABC} under the three Conditions~(A), (B)
  and~(C). 

\begin{rem}
  \label{rem:explicit}
  Obtaining good bounds for the constant $\alpha$ may be important because it is known
  that, for unkilled processes, convergence rates may be better in Wasserstein distance than in total variation (see
  e.g.~\cite[Theorem 2.2]{Malrieu2015} for an example with PDMP).   An explicit expression for the pre-exponential factor $C_1$ can also be obtained from the proof.
\end{rem}

\begin{rem}
  Observe that $C_A\geq 1$ (take $t=0$ in (A)) and $C_C\geq 1$. In particular, $\alpha\leq\gamma_A\frac{\log 2}{\log 2}=\gamma_A$. This means that the rate of convergence we obtain for the penalized process  is lower than the exponential decay rate given by Assumption (A). See Section \ref{sec:almost-sure-contraction} for complementary discussions on this topic.
\end{rem}




The proof of Theorem~\ref{thm:main-under-ABC} relies on four lemmas, stated below, allowing to deduce contraction
  estimates in Wasserstein distance for the following penalized time-inhomogeneous semigroup: 
\[
R^T_{s,t}f(x)=R^{T-s}_{0,t-s}f(x)=\frac{\E_{x} Z_{T-s}f(X_{t-s})}{\E_{x} Z_{T-s}}=\frac{\delta_x P_{t-s}(f
  P_{T-t}\1)}{\delta_x P_{T-s}\1},\quad\forall s\leq t\leq T\in I.
\]
As usual, we define for all probability measure $\mu$ on $E$
\[
\mu R^T_{s,t}f=\int_E R^T_{s,t}f(x)\,\mu(dx).
\]
In particular, $\delta_x R^T_{s,t}f=R^T_{s,t}f(x)$. Recall also that $R_{s,t}^T$ satisfies the (time-inhomogeneous) semigroup property:
\[
R_{r,s}^TR_{s,t}^T=R^T_{r,t},\quad\forall r\leq s\leq t\leq T\in I.
\]
Note that, when $\mu$ is not a Dirac mass, we usually don't have equality between $\frac{\mu P_t}{\mu P_t\mathbbm{1}}$ and
$\mu R_{0,t}^t$.

In all the proof we denote
\[
  \gamma(t)=C_A e^{-\gamma_A t}.
\]
 In the first lemma we extend (A) to penalizations up
to any time $T\geq t$.

\begin{lem}
  \label{lem:A-strong}
  Assumptions~(A) and~(B) imply that for all $x,y\in E$ and $t\leq T\in I$,
  \begin{equation*}
    \E_{(x,y)}\left[G_T^Xd(X_t,Y_t)\right] \leq (1+C_B\bar d)\gamma(t) d(x,y).
  \end{equation*}
\end{lem}

The second lemma shows that $P_t$ and $R_{s,t}^T$ preserve the set of Lipschitz functions.

\begin{lem}
  \label{lem3bis}
  Assume~(A) and~(B). For all $x,y\in E$, all $\phi\in\textnormal{Lip}(d)$ and all $0\leq s\leq t\leq T\in I$,
  \begin{equation}
    \label{eq:P_t-Lipschitz}
    \Big|\delta_x P_{t} \phi-\delta_y P_{t} \phi\Big|
    \leq \left(C_B\|\phi\|_\infty+\gamma(t)\|\phi\|_{\textnormal{Lip}(d)}\right)\, d(x,y) \, \E_x Z_t.
  \end{equation}
  and
  \begin{equation}
    \label{eq:R_t-Lipschitz}
    \Big|\delta_x R^T_{s,t} \phi-\delta_y R^T_{s,t} \phi\Big|
    \leq \left(2C_B\|\phi\|_\infty+(1+C_B \bar d)\gamma(t-s)\|\phi\|_{\textnormal{Lip}(d)}\right) d(x,y).
  \end{equation}
\end{lem}

For all $\kappa>0$, we define the distance $d_\kappa$ on $E$ as
\[
d_\kappa(x,y)=(\kappa d(x,y))\wedge 1,\quad\forall x,y\in E.
\]
The following lemma aim at proving a contraction property of $R_{s,t}^T$ for initial Dirac masses and for a well-chosen distance
$d_\kappa$.
\begin{lem}
  \label{lem4bis} 
  Assume~(A),~(B) and~(C). Setting $\beta=1-\frac{1}{2 C_C}$ and $\kappa=\frac{C_B}{\beta-1/2}$, there exists $t_0>0$ such that, for all $x,y\in E$,  for all
  $\phi\in\textnormal{Lip}(d_\kappa)$ and for all $0\leq s\leq t\leq T\in I$ such that $t-s\geq t_0$,
  \begin{align}
    \label{eq:lem4bis}
    \Big|\delta_x R^T_{s,t} \phi-\delta_y R^T_{s,t} \phi\Big|
    \leq \beta\,\|\phi\|_{\textnormal{Lip}(d_\kappa)} d_\kappa(x,y).
  \end{align}
\end{lem}


The last lemma will be useful to extend the results of Lemma~\ref{lem4bis}  to any initial distribution.

\begin{lem}
  \label{lem:arbitrary-IC}
  Condition (B) implies that, for all $\kappa\geq 1$, all probability measures $\mu$ and $\nu$ on $E$ and all $t\in I$,
  \[
  |\E_\mu Z_t-\E_\nu Z_t|\leq C_B(1+\bar{d}C_B)(\kappa^{-1}\vee\bar{d})\,\mathcal{W}_{d_\kappa}(\mu,\nu)\,\E_\nu Z_t.
  \]
\end{lem}
 We now prove the four lemmas.
\begin{proof}[Proof of Lemma~\ref{lem:A-strong}]
  Assumption~(B) implies that, for all $x,y\in E$ and $t\in I$, $\E_y Z_t\leq (1+C_B\bar d)\E_x Z_t$, hence
  \begin{equation}
    \label{eq:interm-1}
    \sup_{y\in E}\E_y Z_t\leq (1+C_B\bar d)\inf_{y\in E}\E_y Z_t.    
  \end{equation}
  Then, for all $t\leq T\in I$ and $x,y\in E$, it follows from Markov's property that
  \begin{align*}
    \E_{(x,y)}\left[G_T^Xd(X_t,Y_t)\right] & =\frac{\E_{(x,y)}\left[Z_t^Xd(X_t,Y_t)\E_{X_t} Z_{T-t}\right]}{\E_x\left[ Z_t\E_{X_t} Z_{T-t}\right]}
    \\ & \leq\frac{\sup_{y\in E}\E_y Z_{T-t}}{\inf_{y\in E}\E_y Z_{T-t}}\,\E_{(x,y)}\left[G_t^Xd(X_t,Y_t)\right].
  \end{align*}
  The result follows from (A) and~\eqref{eq:interm-1}.
\end{proof}

\begin{proof}[Proof of Lemma~\ref{lem3bis}]
  Let us first notice that
  \begin{multline}
    |\E_{x} Z_{T-s}\phi(X_{t-s})-\E_{y} Z_{T-s}\phi(X_{t-s})| =|\E_{(x,y)} [Z^X_{T-s}\phi(X_{t-s})- Z^Y_{T-s}\phi(Y_{t-s})]| \\
    \leq \|\phi\|_\infty \E_{(x,y)} |Z_{T-s}^X- Z_{T-s}^Y|+\|\phi\|_{\text{Lip}(d)}\E_{(x,y)} Z_{T-s}^X d(X_{t-s},Y_{t-s}).
    \label{eq:pf-lem3bis}
  \end{multline}
  Assuming $s=0$ and $t=T$ and using (A) and (B),~\eqref{eq:P_t-Lipschitz} follows.

  Let us now prove~\eqref{eq:R_t-Lipschitz}. We use the fact that, for all $A,C\in\R$ and all $B,D\in\R\setminus\{0\}$,
  \begin{align}
    \label{eq:calcul-fraction}
    \left|\frac{A}{B}-\frac{C}{D}\right|\leq \frac{|A-C|}{|B|}+\frac{|C|}{|D|}\frac{|D-B|}{|B|}.
  \end{align}
  This entails
  \begin{align*}
    \Big|\delta_x R^T_{s,t} \phi-\delta_y R^T_{s,t} \phi\Big| & \leq \frac{|\E_{x} Z_{T-s}\phi(X_{t-s})-\E_{y} Z_{T-s}\phi(X_{t-s})|}{\E_{x}
      Z_{T-s}}+\|\phi\|_\infty \frac{|\E_{x} Z_{T-s}-\E_{y} Z_{T-s}|}{\E_{x} Z_{T-s}}.
  \end{align*}
  Using~\eqref{eq:pf-lem3bis}, Lemma~\ref{lem:A-strong} and (B), we obtain
  \begin{align*}
    \Big|\delta_x R^T_{s,t} \phi-\delta_y R^T_{s,t} \phi\Big| & \leq 2\|\phi\|_\infty \frac{\E_{(x,y)} |Z_{T-s}^X- Z_{T-s}^Y|}{\E_{x}
      Z_{T-s}}+\|\phi\|_{\text{Lip}(d)}\frac{\E_{(x,y)} Z_{T-s}^X d(X_{t-s},Y_{t-s})}{\E_{x}
      Z_{T-s}} \\ & \leq 2\|\phi\|_\infty C_B d(x,y)+\|\phi\|_{\text{Lip}(d)} (1+C_B\bar d)\gamma(t-s)d(x,y). \qedhere
  \end{align*}
\end{proof}

\begin{proof}[Proof of Lemma~\ref{lem4bis}]
  First observe that $\beta\in(1/2,1)$ so that $\kappa$ is well defined. We can and do assume without loss of generality that $\phi\in\text{Lip}_1(d_\kappa)$. Note that, since $d_\kappa$ is bounded by
  $1$, for all $\phi\in\text{Lip}_1(d_\kappa)$, $\text{osc}(\phi)\leq 1$, so that there exists a constant
  $c$ such that $\|\phi-c\|_\infty\leq\nicefrac12$. Therefore, it is sufficient to prove~\eqref{eq:lem4bis} assuming that
  $\|\phi\|_{\text{Lip}(d_\kappa)}\leq 1$ and $\|\phi\|_\infty\leq\nicefrac12$.

  We start by proving~\eqref{eq:lem4bis} for all $x,y\in E$ such that $d_\kappa(x,y)<1$. Notice  that, since $d_\kappa\leq \kappa d$,
  $\|\phi\|_{\text{Lip}(d)}\leq\kappa\|\phi\|_{\text{Lip}(d_\kappa)}$. Then, for all $x,y\in E$ such that $d_\kappa(x,y)<1$,
  we have $d(x,y)=d_\kappa(x,y)/\kappa$, so it follows from Lemma~\ref{lem3bis} that
  \begin{align}
    \Big|\delta_x R^T_{s,t} \phi-\delta_y R^T_{s,t} \phi\Big| & \leq\left(
      C_B+\gamma(t-s)(1+C_B\bar d)\kappa\|\phi\|_{\text{Lip}(d_\kappa)}\right) d(x,y) \notag \\ &
    =\left(\frac{C_B}{\kappa}+\gamma(t-s)(1+C_B\bar d)\right)d_\kappa(x,y). \label{eq:lem4}
  \end{align}
  Hence, defining $t_1$ as the smallest element of $I$ such that $\gamma(t)\leq 1/(2(1+C_B\bar d))$ for all $t\geq t_1$, we have proved that~\eqref{eq:lem4bis} holds true whenever $d_\kappa(x,y)<1$.
 We now assume that $x$ and $y$ are such that $d_\kappa(x,y)=1$.  We have
  \begin{align*}
    \Big|\delta_x R^T_{s,t} \phi-\delta_y R^T_{s,t} \phi\Big| & =\left|\E_{(x,y)}\left(G_{T-s}^X \phi(X_{t-s})-G_{T-s}^Y
        \phi(Y_{t-s})\right)\right|\\ & \leq \|\phi\|_{\text{Lip}(d_\kappa)}\E_{(x,y)}\left(G_{T-s}^X
                                        d_\kappa(X_{t-s},Y_{t-s})\right)+\|\phi\|_\infty\E_{(x,y)}\left|G_{T-s}^X-G_{T-s}^Y\right|
    \\ & \leq \kappa\E_{(x,y)}\left(G_{T-s}^X
                                        d(X_{t-s},Y_{t-s})\right)+\frac{1}{2}\E_{(x,y)}\left|G_{T-s}^X-G_{T-s}^Y\right|,
  \end{align*}
  since $d_\kappa\leq\kappa d$. Now,
  \begin{align*}
    \E_{(x,y)}\left|G_{T-s}^X-G_{T-s}^Y\right| & =\E_{(x,y)}\left[G_{T-s}^X-G_{T-s}^X\wedge
      G_{T-s}^Y\right]+\E_{(x,y)}\left[G_{T-s}^Y-G_{T-s}^X\wedge G_{T-s}^Y\right] \\ & \leq
    2\left[1-\E_{(x,y)}\left(G_{T-s}^X\wedge G_{T-s}^Y\right)\right].
  \end{align*}
  Therefore, using Lemma~\ref{lem:A-strong} for the first term and 
  Assumption~(C) for the second term, we obtain
  \begin{equation}
    \label{eq:lem5}
    \Big|\delta_x R^T_{s,t} \phi-\delta_y R^T_{s,t} \phi\Big|\leq \kappa(1+C_B\bar d)\bar{d}\gamma(t-s)+1-\frac{1}{C_C}.    
  \end{equation}
  Hence, defining $t_2$ as the first element of $I$ such that $\gamma(t)\leq 1/(2\kappa(1+C_B\bar d)\bar{d}C_C)$ for all $t\geq t_2$, and setting $t_0 = t_1 \vee t_2$
  ends  the proof of Lemma~\ref{lem4bis}.
\end{proof}

\begin{proof}[Proof of Lemma~\ref{lem:arbitrary-IC}]
  Note 
  that $d$ and $d_\kappa$ are equivalent distances, since $d_\kappa\leq \kappa d$ and
  \[
  d(x,y)\leq
  \begin{cases}
    \frac{1}{\kappa}\,d_\kappa(x,y) & \text{if }d_\kappa(x,y)<1, \\
    \bar{d}\,d_\kappa(x,y) & \text{if }d_\kappa(x,y)=1.
  \end{cases}
  \]
  Hence (B) is satisfied with $d$ replaced by $d_\kappa$ and the constant $C_B$ replaced by $C_B (\kappa^{-1}\vee\bar{d})$. In
  particular, for all $x,y\in E$ and $t\in I$,
  \[
    \E_x Z_t\leq \E_y Z_t+C_B(\kappa^{-1}\vee\bar d)d_\kappa(x,y)\E_y Z_t.
  \]
  Integrating this inequality with respect to $\pi$, an optimal coupling of $\mu$ and $\nu$ for the distance $d_\kappa$,
  we obtain
  \begin{align*}
    \E_\mu Z_t & \leq\E_\nu Z_t+C_B (\kappa^{-1}\vee\bar{d})\E_\pi\left[d_\kappa(X_0,Y_0)\E_{Y_0}Z_t\right] \\
     & \leq\E_\nu Z_t+C_B(1+\bar{d}C_B)(\kappa^{-1}\vee\bar{d}) \,(\E_\nu Z_t\wedge\E_\mu Z_t)\,\mathcal{W}_{d_\kappa}(\mu,\nu),
  \end{align*}
  where we used~\eqref{eq:interm-1} 
  in the second inequality. Hence Lemma~\ref{lem:arbitrary-IC} is proved.
\end{proof}

We are now ready to prove Theorem~\ref{thm:main-under-ABC}.

\begin{proof}[Proof of Theorem~\ref{thm:main-under-ABC}] The proof is divided into three steps.

  \medskip
  \noindent\emph{Step 1. Proof of~\eqref{eq:main-compact} with Dirac initial distributions.}

  \noindent Recall the Kantorovich-Rubinstein formula: for all probability measures $\mu$ and $\nu$ on $E$,
  \[
  \mathcal{W}_{d_\kappa}(\mu,\nu)=\sup_{\phi\in\text{Lip}_1(d_\kappa)}\left|\int_E \phi(x)\mu(dx)-\int_E\phi(y)\nu(dy)\right|.
  \]
  It then follows from Lemmas~\ref{lem4bis}  that, for all $\phi\in\text{Lip}_1(d_\kappa)$, all $x,y\in E$ and all
  $T\in I$ and $k\in\mathbb{Z}_+$ such that $(k+1)t_0\leq T$,
  \begin{equation}
    \label{eq:pf-main-thm-2}
    \left|R_{k t_0,(k+1)t_0}^T \phi(x)-R_{k t_0,(k+1)t_0}^T \phi(y)\right|\leq \beta d_\kappa(x,y).
  \end{equation}
  Taking the supremum over $\phi\in\text{Lip}_1(d_\kappa)$, this entails
  \[
    \mathcal{W}_{d_\kappa}(\delta_x R^T_{kt_0,(k+1)t_0}, \delta_y R^T_{kt_0,(k+1)t_0})\leq \beta d_\kappa(x,y).
  \]
  Assuming $k\geq 1$ and integrating~\eqref{eq:pf-main-thm-2} with respect to $\pi(dx,dy)$, an optimal coupling between $\delta_x
  R_{(k-1)t_0,kt_0}^T$ and $\delta_y R_{(k-1)t_0,kt_0}^T$, we obtain
  \begin{align*}
    \left|R_{(k-1) t_0,(k+1)t_0}^T \phi(x)-R_{(k-1) t_0,(k+1)t_0}^T \phi(y)\right| & \leq \beta \iint_{E^2} d_\kappa(x,y)\pi(dx,dy)
    \\ & =\beta\mathcal{W}_{d_\kappa}\left(\delta_x
    R_{(k-1)t_0,kt_0}^T,\delta_y R_{(k-1)t_0,kt_0}^T\right)
  \end{align*}
  Taking again the supremum over $\phi\in\text{Lip}_1(d_\kappa)$ and proceeding by (decreasing) induction over $k$, we deduce that, for all
  $k\in\Z_+$ such that $kt_0\leq T$,
  \[
  \mathcal{W}_{d_\kappa}(\delta_x R_{0,kt_0}^T,\delta_y R_{0,kt_0}^T)\leq \beta^k d_\kappa(x,y).
  \]
  Now, we deduce from the inequalities~\eqref{eq:lem4} and~\eqref{eq:lem5} obtained in the proof of Lemma~\ref{lem4bis}, that, setting
  $C_0=\left(\frac{1}{2}(1-\frac{1}{C_C})+(1+C_B\bar d)\gamma(0)\right)\vee\left(1-\frac{1}{C_C}+\frac{2 C_B\bar d}{1-1/C_C}\gamma(0)\right)$, we have for all
  $\phi\in\text{Lip}_1(d_\kappa)$, all $x,y\in E$ and all $s< t\leq T$ such that $t-s\leq t_0$,
  \[
  \left|\delta_xR_{s,t}^T \phi-\delta_yR_{s,t}^T \phi\right|\leq C_0 d_\kappa(x,y),
  \]
  so that $R_{s,t}^T\phi\in\text{Lip}(d_\kappa)$ with $\|R_{s,t}^T\phi\|_{\text{Lip}(d_\kappa)}\leq C_0$.
  This and the previous inequality entails that, for all $\phi\in\text{Lip}_1(d_\kappa)$ and all $t\leq T\in I$, setting $k$ as the
  integer part of $t/t_0$,
  \begin{align*}
    \left|\delta_x R_{0,t}^T\phi-\delta_y R_{0,t}^T\phi\right| =\left|\delta_x R_{0,kt_0}^TR_{kt_0,t}^T\phi-\delta_y
      R_{0,kt_0}^TR_{kt_0,t}^T\phi\right|  & \leq C_0\mathcal{W}_{d_\kappa}(\delta_x R_{0,kt_0}^T,\delta_y R_{0,kt_0}^T) \\ & \leq C_0\beta^k d_\kappa(x,y).
  \end{align*}
  Hence, setting $\alpha=\frac{\log (1/\beta)}{t_0}$, we deduce that, for all $t\leq T$ and $x,y\in E$,
  \begin{align}
    \label{eq:QE}
    \mathcal{W}_{d_\kappa}(\delta_x R_{0,t}^T,\delta_y R_{0,t}^T)\leq C_0\,e^{\log\beta\left(\frac{t}{t_0}-1\right)} d_\kappa(x,y)=\frac{C_0}{\beta}
    e^{-\alpha t} d_\kappa(x,y).
  \end{align}
  Note that, since $\gamma(t)=C_A e^{-\gamma_A t}$, we obtain~\eqref{eq:alpha-explicite} from the definitions of $t_1$ and $t_2$ in the proof of
  Lemma~\ref{lem4bis} noticing that
  $e^{\gamma_A t_0}=2C_A(1+C_B\bar d)\left[1\vee\frac{2 C_B C_C\bar d}{1-1/C_C}\right]$.
  
  \medskip
  \noindent\emph{Step 2. Extension of~\eqref{eq:main-compact} to arbitrary initial distributions.}

  \noindent Let $\mu$ and $\nu$ be two probability measures on $E$. For all $t\in I$, we define
  \begin{equation}
  \label{eq:defmu't}
  \mu'_t(f)=\frac{\E_\mu[Z_t f(X_0)]}{\E_\mu Z_t}\quad\text{and}\quad \nu'_t(f)=\frac{\E_\nu[Z_t f(X_0)]}{\E_\nu Z_t}.
  \end{equation}
  Then, for all $t\leq T\in I$,
  \begin{equation}
    \label{eq:pf-main-thm-3}
    \mu'_T R_{0,t}^T f=\frac{\int_E R_{0,t}^T f(x)\,\E_x Z_T\,\mu(dx)}{\E_\mu Z_T} = \frac{\int_E \E_{x}[f(X_t)Z_T]\,\mu(dx)}{\E_\mu
      Z_T}=\frac{\E_\mu[f(X_t)Z_T]}{\E_\mu Z_T}
  \end{equation}
  and similarly for $\nu'_T R_{0,t}^T$.

  Let $\pi'$ be an optimal coupling of the measures $\mu'_T$ and $\nu'_T$. We deduce from~\eqref{eq:QE} that, for all
  $\phi\in\text{Lip}_1(d_\kappa)$,
  \begin{align*}
    \left|\mu'_T R^T_{0,t}\phi-\nu'_T R^T_{0,t}\phi\right|
    & \leq\iint_{E^2}\left|\delta_x R^T_{0,t}\phi-\delta_y R^T_{0,t}\phi\right|\pi'(dx,dy) \\
    & \leq\iint_{E^2}\mathcal{W}_{d_\kappa}(\delta_x R^T_{0,t},\delta_y R^T_{0,t})\pi'(dx,dy) \leq
                                \frac{C_0}{\beta}\,e^{-\alpha t}\,\mathcal{W}_{d_\kappa}(\mu'_T,\nu'_T).
  \end{align*}
  Hence, we deduce from~\eqref{eq:pf-main-thm-3} that
  \[
  \mathcal{W}_{d_\kappa}\left(\frac{\mu P_t(\cdot P_{T-t}\mathbbm{1})}{\mu P_T\mathbbm{1}},\frac{\nu P_t(\cdot P_{T-1}\mathbbm{1})}{\nu
      P_T\mathbbm{1}}\right)
  \leq \frac{C_0}{\beta}\, e^{-\alpha t}\,\mathcal{W}_{d_\kappa}(\mu'_T,\nu'_T).
  \]
  Now, for all $\phi\in\text{Lip}_1(d_\kappa)$, assuming without loss of generality that $\|\phi\|_\infty\leq 1/2$ and denoting by $\pi$ an
  optimal coupling of $\mu$ and $\nu$, we have
  \begin{align*}
    \left|\mu'_T(\phi)\right.
    & \left.-\nu'_T(\phi)\right| \leq\iint_{E^2}\pi(dx,dy)\left|\frac{\E_x Z_T}{\E_\mu Z_T}\phi(x)-\frac{\E_y Z_T}{\E_\nu Z_T}\phi(y)\right| \\
    & \leq\iint_{E^2}\pi(dx,dy)\, |\phi(x)-\phi(y)|\, \frac{\E_x Z_T}{\E_\mu
      Z_T}+\iint_{E^2}\pi(dx,dy)\,|\phi(y)|\,\left|\frac{\E_x Z_T}{\E_\mu Z_T}-\frac{\E_y Z_T}{\E_\nu Z_T}\right| \\
      & \leq\iint_{E^2}\pi(dx,dy)d_\kappa(x,y)\frac{\E_x Z_T}{\E_\mu Z_T} \\ & \qquad
    +\frac{1}{2}\iint_{E^2}\pi(dx,dy)\left(\frac{|\E_x Z_T-\E_y Z_T|}{\E_\mu Z_T}+\frac{\E_y Z_T}{\E_\mu Z_T\E_\nu Z_T}|\E_\nu 
      Z_T-\E_\mu Z_T|\right),
  \end{align*}
  where we used~\eqref{eq:calcul-fraction} in the last inequality. We then use Lemma~\ref{lem:arbitrary-IC},
  Inequality~\eqref{eq:interm-1} and the fact that (B) is satisfied with $d$ replaced by $d_\kappa$ and $C_B$ replaced by
  $C_B(\kappa^{-1}\vee\bar{d})$ (see the proof of Lemma~\ref{lem:arbitrary-IC}) to obtain
  \begin{align*}
    \left|\mu'_T(\phi)-\nu'_T(\phi)\right| & \leq
    \left(1+\bar{d}C_B+\frac{(1+\bar{d}C_B)\,C_B(\kappa^{-1}\vee\bar{d})}{2}\right)\,\mathcal{W}_{d_\kappa}(\mu,\nu)
    \\ & \qquad+\frac{C_B(1+\bar{d}C_B)^2(\kappa^{-1}\vee\bar{d})}{2}\mathcal{W}_{d_\kappa}(\mu,\nu).
  \end{align*}
  Therefore, setting
  \[
  C_1=\frac{C_0(1+\bar{d}C_B)(1\vee(\kappa\bar{d}))}{\beta}\left(1+C_B\left(1+\frac{\bar{d}C_B}{2}\right)(\kappa^{-1}\vee\bar{d})\right),
  \]
  we have proved that for all $t\leq T\in I$ and all probability measures $\mu$ and $\nu$ on $E$,
  \[
  \mathcal{W}_{d_\kappa}\left(\frac{\mu P_t(\cdot P_{T-t}\mathbbm{1})}{\mu P_T\mathbbm{1}},\frac{\nu P_t(\cdot P_{T-1}\mathbbm{1})}{\nu
      P_T\mathbbm{1}}\right)
  \leq \frac{C_1}{1\vee(\kappa\bar{d})}\, e^{-\alpha t}\,\mathcal{W}_{d_\kappa}(\mu,\nu).
  \]
  Setting $T=t$ and using that $\kappa^{-1} d_\kappa\leq d\leq (\kappa^{-1}\vee\bar{d}) d_\kappa$,~\eqref{eq:main-compact} follows.

  \medskip
  \noindent\emph{Step 3. Existence and uniqueness of a quasi-stationary distribution.}

  \noindent Uniqueness of a quasi-stationary distribution easily follow from~\eqref{eq:main-compact}, therefore we only have to prove
  existence. Since $(E,d)$ is Polish, the set of probability measures $\mu$ on $E$ such that $\int_E d(x,x_0)\,\mu(dx)<\infty$ for
  some (hence all) $x_0\in E$ equiped with the distance $\mathcal{W}_d$ is also Polish (see \cite[Theorem 6.18]{Villani2009}). Here,
  $d$ is bounded, so the set of probability measures on $E$ is complete for $\mathcal{W}_d$. Let $t_0$ such that $C_1 e^{ - \alpha t_0} < 1$. Then, from~\eqref{eq:main-compact}, the map $\mu \mapsto \frac{\mu P_{t_0} }{\mu P_{t_0} \mathbbm{1}}$ is a contraction on the set of probability measure equipped with $\mathcal{W}_d$. Therefore, by the Banach fixed point theorem, there exists a unique distribution $\nu$ on $E$ such that
  \[
  \frac{\nu P_{t_0} }{\nu P_{t_0} \mathbbm{1}} = \nu.
  \]
 Now, let $t \geq 0$. We have
 \[
   \frac{\frac{\nu P_{t} }{\nu P_{t} \mathbbm{1}} P_{t_0}}{\frac{\nu P_{t} }{\nu P_{t} \mathbbm{1}} P_{t_0} \mathbbm{1}}= \frac{\nu
     P_{t + t_0}}{\nu P_{ t + t_0} \mathbbm{1}}= \frac{\frac{\nu P_{t_0}}{\nu P_{t_0} \mathbbm{1}} P_t}{\frac{\nu P_{t_0}}{\nu
       P_{t_0} \mathbbm{1}} P_t \mathbbm{1}} = \frac{\nu P_t}{\nu P_t \mathbbm{1}}
 \]
This proves that $\frac{\nu P_{t} }{\nu P_{t} \mathbbm{1}}$ is a fixed point of $\mu \mapsto \frac{\mu P_{t_0} }{\mu P_{t_0} \mathbbm{1}}$, which implies by uniqueness that
\[ 
 \frac{\nu P_{t} }{\nu P_{t} \mathbbm{1}} = \nu, \quad \forall t \geq 0
  \]
 i.e., that $\nu$ is a QSD.  
   Hence the proof of Theorem~\ref{thm:main-under-ABC} is completed.
\end{proof}

\subsection{End of the proof of Theorem~\ref{thm:main-compact}: (A) implies the additional Conditions~(B) and~(C)}
\label{sec:main-proof-2}

  The goal of this section is to prove Proposition~\ref{prop:Aexp-implies-H}, i.e. (A) implies (H) and the next result.
\begin{prop}
  \label{thm:equiv-A-B-Aexp}
  Conditions (A) and (H) imply Conditions~(B) and~(C).
\end{prop}
\noindent Together with Theorem~\ref{thm:main-under-ABC}, these two propositions imply Theorem~\ref{thm:main-compact}.

In all this Subsection, we give the proofs in the case $I=\mathbb{R}_+$. The proofs are similar in the case $I=\mathbb{Z}_+$ and we
leave the details to the reader.

\subsubsection{Proof of Proposition \ref{prop:Aexp-implies-H}:  (A) implies (H)}
\label{sec:proofofABC}

We make use of the following lemma:
\begin{lem}
  \label{lem:simple}
  Let $f,g:[0,+\infty)\to\mathbb R$ be two $L^1$ functions. Then, for all $t\geq 0$, we have
  \begin{align*}
    |e^{\int_0^t f_s \mathrm ds}-e^{\int_0^t g_s \mathrm ds}|\leq \int_0^t e^{\int_0^s f_u \mathrm du} |f_s-g_s| e^{\int_s^t g_u \mathrm du}\mathrm ds
  \end{align*}
\end{lem}

\begin{proof}
  It is sufficient to prove it for continuous functions $f,g$, since the result follows by a standard density argument. In order to
  prove this result, we consider the differentiable function
  \[
    H:t\in[0,+\infty)\mapsto \int_0^t (f_s-g_s)\,\mathrm ds.
  \]
  We have
  \[
    e^{\int_0^t (f_s-g_s) \mathrm ds}-1 =e^{H_t}-1 = \int_0^t H'_s\,e^{H_s}\,\mathrm ds \leq \int_0^t |f_s-g_s|\,e^{\int_0^s (f_u-g_u)\mathrm du}\,\mathrm ds.
  \]
  Multiplying both sides of the inequality by $e^{\int_0^t g_s ds}$ concludes the proof.
\end{proof}

We now proceed to the proof of Proposition \ref{prop:Aexp-implies-H}
Let us introduce, for all $t\geq 0$,
    \[
    H_t:=\sup_{x\neq y,s\leq t}\frac{\E_x Z_s}{\E_y Z_s}\leq e^{\text{osc}(\rho)\,t}.
    \]
    Using Lemma~\ref{lem:simple} and Markov's inequality, we obtain, for any $t\geq t_0\geq 0$,
    \begin{align*}
        \E_xZ_t& \leq \E_y(Z_t)+\|\rho\|_{\text{Lip}(d)}\int_0^t\E_{(x,y)}\left[Z^X_s
        d(X_s,Y_s)\,\E_{Y_s}Z_{t-s}\right]ds \\ 
        &\leq \E_y(Z_t)+\|\rho\|_{\text{Lip}(d)}\int_0^{t_0}\E_{(x,y)}\left[Z^X_s
        d(X_s,Y_s)\,\E_{Y_s}Z_{t-s}\right]ds\notag\\
        &\qquad +\|\rho\|_{\text{Lip}(d)}\int_{t_0}^{t}\E_{(x,y)}\left[Z^X_s
        d(X_s,Y_s)\,\E_{Y_s}Z_{t-s}\right]ds. 
    \end{align*}
    For the first integral in the right hand side, we observe that
    \begin{align}
        \label{eq:inequality1Aexp}
        \int_0^{t_0}\E_{(x,y)}\left[Z^X_s
        d(X_s,Y_s)\,\E_{Y_s}Z_{t-s}\right]ds\leq \bar{d}\, e^{\text{osc}(\rho) t_0} \E_y(Z_t).
    \end{align}
    For the second integral, we deduce from (A) that
    \begin{align*}
        \E_{(x,y)}\left[Z^X_s
        d(X_s,Y_s)\,\E_{Y_s}Z_{t-s}\right]
        & \leq \E_{(x,y)}\left[Z^X_s
        d(X_s,Y_s)\,\right] H_{t-s}\min_{z\in E}\E_{z}Z_{t-s}\\
        & \leq {C_A e^{-\gamma_A  s}}\,\bar{d}\, \E_x(Z_s) H_{t-s} \, \E_y\Big(\frac{Z_{t-s}\E_{Y_{t-s}} Z_s}{\min_{z\in E}\E_z
          Z_s}\Big)\\
        & \leq {C_A e^{-\gamma_A  s}} \,\bar{d}\,  H_s H_{t-s}\, \E_y Z_t.
    \end{align*}
    We deduce that
    \begin{align}
    \label{eq:pivot}
        H_t \leq 1+\bar{d}\, \|\rho\|_{\text{Lip}(d)} e^{\text{osc}(\rho) t_0}+ \|\rho\|_{\text{Lip}(d)} {C_A}\,\bar{d} \int_{t_0}^t {e^{-\gamma_A  s}} H_s H_{t-s}\,\mathrm ds.
    \end{align}
    Now let $a_0\leq \text{osc}(\rho)$ and assume that $H_t\leq C_0 e^{a_0t}$ for all $t\geq 0$ for some constant $C_0>0$ (this is
    true for $a_0=\text{osc}(\rho)$, our goal is to prove that it is true for $a_0=0$). Then
    \begin{align*}
     H_t &\leq 1+\bar{d}\, \|\rho\|_{\text{Lip}(d)} e^{\text{osc}(\rho) t_0}+ \|\rho\|_{\text{Lip}(d)} C_\gamma \,\bar{d}\, C_0^2 e^{a_0t}\int_{t_0}^t e^{-\gamma  s}\,\mathrm ds\\
         &\leq 1+\bar{d}\, \|\rho\|_{\text{Lip}(d)} e^{\text{osc}(\rho) t_0}+ \frac{\|\rho\|_{\text{Lip}(d)} C_\gamma \,\bar{d}\, C_0^2}{\gamma}\, e^{a_0t} e^{-\gamma  t_0}.
    \end{align*}
    Now setting $a_1=a_0\,\text{osc}(\rho)/(\text{osc}(\rho)+\gamma)$ and choosing $t_0=a_1 t/\text{osc}(\rho)\leq t$, we obtain
    \begin{align*}
    H_t &\leq (1+\bar{d}\, \|\rho\|_{\text{Lip}(d)}) e^{a_1 t}+ \frac{\|\rho\|_{\text{Lip}(d)} C_\gamma\, \bar{d}\,C_0^2}{\gamma} e^{a_1t},
    \end{align*}
    so that there exists a constant $C_1$ such that, for all $t\geq 0$, $H_t\leq C_1 e^{a_1 t}$. Proceeding by induction, with deduce that, for all $n\geq 0$, there exists a constant $C_n\geq 0$ such that
    \begin{align*}
    H_t \leq C_n e^{a_n t},\text{ where } a_n=\text{osc}(\rho)\left(\frac{\text{osc}(\rho)}{\text{osc}(\rho)+\gamma}\right)^n.
    \end{align*}
    In particular, there exists $C_H>0$ such that
    \begin{align*}
    H_t\leq C_H e^{\gamma t/2}.
    \end{align*}
    We can thus find $t_0$ large enough so that
    \[
     b_0:=\|\rho\|_{\text{Lip}(d)} C_\gamma\, \bar{d}\int_{t_0}^{+\infty} e^{-\gamma  s} H_s\,\mathrm ds < 1.
    \]
    For this value of $t_0$,~\eqref{eq:pivot} implies that, for all $t\geq 0$,
    \begin{align}
    H_t \leq 1+\bar{d}\, \|\rho\|_{\text{Lip}(d)} e^{\text{osc}(\rho) t_0}+ b_0 H_t.
    \end{align}
    This proves that $\sup_{t\geq 0} H_t$ is finite.

\subsubsection{First part of the proof of Proposition~\ref{thm:equiv-A-B-Aexp}: (A) and (H) imply (B)}
\label{sec:main-proof-1}

We start by proving a stronger version of (A): for all $x,y\in E$ and $t\leq T$,
\begin{equation}
\label{Aexpstrong-v2}
\E_{(x,y)}\left[G_T^Xd(X_t,Y_t)\right] \leq C_AC_H e^{ - \gamma_A t} d(x,y).
\end{equation}
Indeed, for all $t\leq T\in I$ and $x,y\in E$, it follows from Markov's property that
\begin{align*}
  \E_{(x,y)}\left[G_T^Xd(X_t,Y_t)\right] & =\frac{\E_{(x,y)}\left[Z_t^Xd(X_t,Y_t)\E_{X_t} Z_{T-t}\right]}{\E_x\left[ Z_t\E_{X_t} Z_{T-t}\right]}
  \\ & \leq\frac{\sup_{y\in E}\E_y Z_{T-t}}{\inf_{y\in E}\E_y Z_{T-t}}\,\E_{(x,y)}\left[G_t^Xd(X_t,Y_t)\right].
\end{align*}
Now, using the fact that $|e^a-e^b|\leq|a-b|(e^a\vee e^b)\leq |a-b|(e^a+e^b)$,we have
\begin{align*}
\mathbb E_{(x,y)}\left| Z_t^X-Z_t^Y\right|
& \leq\mathbb E_{(x,y)}\Big[\left(\exp \left(\int_0^t\rho(X_s)ds\right) + \exp
\left(\int_0^t\rho(Y_s)ds\right)\right)\\
&\qquad\qquad\qquad\qquad\qquad\qquad\qquad\qquad\times\left|\int_0^t (\rho(X_s)-\rho(Y_s))ds\right|\Big] \\ 
&\leq\|\rho\|_{\text{Lip}(d)}\int_0^t\mathbb E_{(x,y)}\left[\left(Z^X_{t}+Z^Y_{t}\right)d(X_s,Y_s)\right]ds \\
& \leq  \|\rho\|_{\text{Lip}(d)} C_AC_H\,\int_0^t e^{ - \gamma_A s} \,\mathrm ds \,\left(\mathbb E_{x} Z_t + \mathbb E_{y} Z_t\right) \, d(x,y) \\
& \leq  \|\rho\|_{\text{Lip}(d)}\, \frac{C_AC_H}{\gamma_A}(1 + C_H) \mathbb E_{y} Z_t d(x,y),
\end{align*}
where we have used \eqref{Aexpstrong-v2} in the third inequality, and (H) in the fourth one.  This implies that, for $C_B = \frac{C_AC_H}{\gamma_A}(1 + C_H)$, we have, for all $x, y \in E$ and $t \in I$, 
\begin{equation}
\label{B'}
\mathbb E_{(x,y)}\left| Z_t^X-Z_t^Y\right| \leq C_B\mathbb{E}_{y} Z_t d(x,y).
\end{equation}

\subsubsection{ End of the proof of Proposition~\ref{thm:equiv-A-B-Aexp}: (A) and (H) imply (C)}
\label{sec:main-proof-4}

By Condition~(H), for all $x,y\in E$ and $t \geq 0$,
\[
\E_{(x,y)} ( G_t^X \wedge G_t^y ) \geq \frac{1}{C_H\E_x(Z_t)} \E_{(x,y)} (Z_t^X \wedge Z_t^Y),
\]
and (C) is proved if we prove that $ \E_{(x,y)} (Z_t^X \wedge Z_t^Y) \geq c \E_x (Z_t^X)$ for some constant $c > 0$.

We first prove that (C) holds true for close starting points $x,y$. Indeed, using Condition~(B), we get for all $(x,y)$ such that $2C_B d(x,y) \leq 1$,
\begin{align*}
 \E_{(x,y)} (Z_t^X \wedge Z_t^Y) &  = \frac{1}{2}\left( \E_x Z_t + \E_y Z_t  - \E_{(x,y)}(| Z_t^X - Z_t^Y|) \right)\\
 & \geq \frac{1}{2}\left(  \E_x Z_t  +(1 - C_B d(x,y))\E_y Z_t  \right)\\
 & \geq \frac{1}{2}\left(  \E_x Z_t  + \frac{1}{2}\E_y Z_t  \right) \geq \frac{1 + (2C_H)^{-1}}{2} \E_x Z_t
\end{align*}
and thus (C) is satisfied whenever $2C_B d(x,y) \leq 1$, with constant $C_C' := \frac{1 + (2C_H)^{-1}}{2}$.

Now, let $x,y \in E$, $t_0 = \frac{1}{\gamma_A} \log( 4 C_B C_A C_H \bar d)$ and $t \geq t_0$. Using the bound we derived for points that are at distance lower than $1/2C_B$, and the fact that $(X,Y)$ is Markov, we have
\begin{align*}
 \E_{(x,y)} (Z_t^X \wedge Z_t^Y) & \geq e^{ - osc(\rho)t_0} \E_{(x,y)} \left( Z_{t_0}^X \E_{X_{t_0}, Y_{t_0}}( Z_{t-{t_0}}^X \wedge Z_{t-t_0}^Y) \1_{ d(X_{t_0},Y_{t_0}) \leq \frac{1}{2C_B}} \right)\\
 & \geq e^{ - osc(\rho)t_0} C'_C \E_{(x,y)}\left( Z_{t_0}^X \E_{X_{t_0}}(Z_{t-t_0}^X) \1_{ d(X_{t_0},Y_{t_0}) \leq \frac{1}{2C_B'}} \right)\\
 & =e^{ - osc(\rho)t_0} C'_C \E_{(x,y)}\left( Z_{t}^X  \1_{ d(X_{t_0},Y_{t_0}) \leq \frac{1}{2C_B}} \right)
\end{align*}
Now, \eqref{Aexpstrong-v2} yields for all $t\geq t_0$
\begin{align*}
\E_{(x,y)}\left( Z_{t}^X  \1_{ d(X_{t_0},Y_{t_0}) > \frac{1}{2C_B}} \right) & \leq 2C_B \E_{(x,y)} Z_t^X  d(X_{t_0}, Y_{t_0} )
\leq 2 C_B C_A C_H e^{ -\gamma_A t_0} \, \bar d \, \E_x Z_t =\frac{\E_x Z_t}{2}.
\end{align*}
Therefore, 
\[
\E_{(x,y)} (Z_t^X \wedge Z_t^Y)  \geq \frac{e^{ - osc(\rho)t_0} C'_C }{2} \E_x(Z_t)
\]
which proves (C) for $t \geq t_0$. Finally, for all $t \leq t_0$,
\[
\E_{(x,y)} ( G_t^X \wedge G_t^y ) \geq e^{ - osc(\rho) t_0}
\]
and (C) holds true. 

\subsection{Subexponential penalized coupling}
\label{sec:subexp}

 For non-penalized process, it is known that uniform contraction in Wasserstein distance implies uniform exponential
  convergence. It is therefore natural to study the consequences of the following subexponential counterpart to Condition~(A):
  \begin{description}
\item[(A')] (subexponential penalized coupling)  There exists a nonincreasing function $\gamma:I\rightarrow \R_+$ converging
  to 0 and for all $x,y\in E$ there exists a Markovian coupling  $\P_{(x,y)}$ between $\mathbb{P}_x$ and $\mathbb{P}_y$ such
  that for all $t\in I$,
 \begin{equation*}
    \E_{(x,y)}\left[G^X_{t}d(X_t,Y_t)\right]\leq \gamma(t) d(x,y). 
  \end{equation*}
\end{description}

Note that we never used in the proof of Theorem~\ref{thm:main-under-ABC} the explicit expression of the function $\gamma(t)$, except
to compute the explicit expression of $\alpha$ in~\eqref{eq:alpha-explicite}. Therefore, Theorem~\ref{thm:main-under-ABC} remains true
under the weaker conditions~(A'), (B) and~(C).

Furthermore, the next result precises the link between Conditions~(A) and~(A').
\begin{prop}
  \label{prop:subexp}
  Conditions~(A') and~(H) imply Condition~(A).
\end{prop}
Note that this result can be combined with~\eqref{eq:interm-1} and Propositions~\ref{prop:Aexp-implies-H}
and~\ref{thm:equiv-A-B-Aexp} to deduce the following equivalences:
\[
  (A)\quad\Longleftrightarrow\quad (A')\text{ and }(H) \quad\Longleftrightarrow\quad (A')\text{ and }(B).
\]
In addition, as proved in Proposition~\ref{thm:equiv-A-B-Aexp}, all these conditions imply (C).

\begin{proof}[Proof of Proposition~\ref{prop:subexp}]

From (A'), let $t_0\in[0,+\infty)$ be large enough so that, for all $t_1\geq t_0$,
\[
\E_{(x,y)}\left[ G^X_{t_1} d(X_{t_1},Y_{t_1})\right]\leq\frac{1}{2 C_H}d(x,y).
\]
Then, for all $k\in\N$, we deduce from (H) that
\begin{align*}
  \E_{(x,y)}\left[ G^X_{kt_1} d(X_{t_1},Y_{t_1})\right]
  & =\E_{(x,y)}\left[ G^X_{t_1} d(X_{t_1},Y_{t_1})\frac{\E_x (Z_{t_1}) \E_{X_{t_1}} (Z_{(k-1)t_1})}{\E_x \left[ Z_{t_1}\E_{X_{t_1}}Z_{(k-1)t_1}\right]}\right] \\
  & \leq C_H\E_{(x,y)}\left[ G^X_{t_1} d(X_{t_1},Y_{t_1})\right]\leq \frac{1}{2}d(x,y).
\end{align*}
Then, for all $1\leq \ell\leq k$,
\begin{align*}
  \E_{(x,y)}\left[ Z^X_{kt_1} d(X_{\ell t_1},Y_{\ell t_1})\right]
  & =\E_{(x,y)}\left[ Z^X_{(\ell-1)t_1}\E_{(X_{(\ell-1)t_1},Y_{(\ell-1)t_1})}\left(Z^X_{(k-\ell+1)t_1} d(X_{t_1},Y_{t_1})\right)\right] \\
  & \leq \frac{1}{2}\E_{(x,y)}\left[
    Z^X_{(\ell-1)t_1}d(X_{(\ell-1)t_1},Y_{(\ell-1)t_1})\E_{(X_{(\ell-1)t_1},Y_{(\ell-1)t_1})}Z^X_{(k-\ell+1)t_1}\right] \\
  & =\frac{1}{2}\E_{(x,y)}\left[ Z^X_{kt_1} d(X_{(\ell-1) t_1},Y_{(\ell-1) t_1})\right],
\end{align*}
from which we deduce by induction that, for all $k\in\N$,
\begin{equation}
  \label{eq:AH-Aexp}
 \E_{(x,y)}\left[ Z^X_{kt_1} d(X_{k t_1},Y_{k t_1})\right]\leq 2^{-k} d(x,y)\E_x Z_{k t_1}.
\end{equation}
Now, let $t\geq t_0$ be arbitrary and set $t_1:=t/k$ with $k:=\lfloor t/t_0\rfloor\geq 1$. Note that $t_1\geq t_0$, so that it follows
from~\eqref{eq:AH-Aexp} that
\[
  \E_{(x,y)}\left[ G^X_{t} d(X_{t},Y_t)\right]=\E_{(x,y)}\left[ G^X_{kt_1} d(X_{k t_1},Y_{k t_1})\right]\leq 2^{1-t/t_0} d(x,y).
\]
Since it follows from (A') that, for all $t\leq t_0$,
\[
  \E_{(x,y)}\left[ G^X_{t} d(X_{t},Y_t)\right]\leq \gamma(0) d(x,y),
\]
we have proved (A).
\end{proof}

\subsection{Proof of the other consequences}
\label{sec:other-proofs}

\subsubsection{Proof of Theorem \ref{thm:eta-compact}}
\label{sec:proof-eta}

For all $t\geq 0$ and $x\in E$, we define
\begin{equation}
  \label{eq:def-eta_t}
  \eta_t(x)=e^{\lambda_0 t}\E_x Z_t=\frac{\E_x Z_t}{\E_{\nu_{QS}} Z_t}.  
\end{equation}
Markov's property entails that, for all $s,t\geq 0$ and $x\in E$, $\eta_{t+s}(x)=e^{\lambda_0 t}P_t \eta_s(x)$, hence
\begin{equation}
  \label{eq:basic-eta_t}
  |\eta_{t+s}(x)-\eta_t(x)|=\eta_t(x)\left|\frac{P_t\eta_s(x)}{P_t\mathbbm{1}(x)}-1\right|.
\end{equation}
It follows from~\eqref{eq:interm-1} that $\E_x Z_t\leq(1+\bar{d}\,C_B)\E_{\nu_{QS}} Z_t=(1+\bar{d}C_B)e^{-\lambda_0 t}$, so
we deduce from (B) that $\eta_t$ is $d$-Lipschitz with $\|\eta_t\|_{\text{Lip}(d)}\leq
C_B(1+\bar{d}C_B)$ for all $t\in I$. Hence we deduce from~\eqref{eq:hyp-eta} that
\begin{align*}
  \left|\frac{P_t\eta_s(x)}{P_t\mathbbm{1}(x)}-1\right|
  & =\left|\delta_x R_{0,t}^t \eta_s-\nu_{QS}(\eta_s)\right| \\
  & \leq C_1C_B(1+\bar{d}C_B)\,e^{-\alpha t}\,\mathcal{W}_{d}(\delta_x,\nu_{\text{QSD}}) \\
  & \leq C_1C_B\bar{d}(1+\bar{d}C_B)\,e^{-\alpha t}.
\end{align*}
In addition, since $\nu_{QS}(\eta_t)=1$ and $d\leq \bar d$, the Lipschitz bound on $\eta_t$ implies that
$\|\eta_t\|_\infty\leq 1+C_B\bar{d}(1+\bar{d}C_B)$, so we deduce from~\eqref{eq:basic-eta_t} that
\begin{align*}
  |\eta_{t+s}(x)-\eta_t(x)|
    & \leq C_1C_B\bar{d}(1+\bar{d}C_B)\left(1+C_B\bar{d}(1+\bar{d}C_B)\right)\,e^{-\alpha t}.
\end{align*}
Therefore, $\eta_t$ converges uniformly to a function $\eta$. Since $\eta_t$ is uniformly (in $t$) $d$-Lipschitz and
$\nu_{\text{QSD}}(\eta_t)=1$ for all $t\in I$, the same properties are satisfied by $\eta$. Since in addition
$P_{t}\eta_s(x)=e^{-\lambda_0 t}\eta_{t+s}(x)$, we deduce that $P_t\eta(x)=e^{-\lambda_0 t}\eta(x)$ letting $s\rightarrow+\infty$.
Finally, it follows from~\eqref{eq:interm-1} and the fact that $\nu_{\text{QSD}}(\eta_t)=1$ that, for all $t\in I$,
\[
  1\leq\sup_{x\in E}\eta_t(x)\leq (1+\bar{d}C_B)\inf_{y\in E}\eta_t(y).
\]
Letting $t\rightarrow+\infty$, we deduce that $\eta(x)\geq 1/(1+\bar{d}C_B)>0$ for all $x\in E$.

\subsubsection{Proof of Theorem \ref{thm:Q-process-compact}}
\label{sec:proof-Q-proc}

Using Theorem~\ref{thm:eta-compact} and the notation~\eqref{eq:def-eta_t}, we have for all $s\leq t\in I$ and
$A\in\mathcal{F}_s$
\begin{align}
  \left|\frac{\E_x(\mathbbm{1}_A Z_t)}{\E_x Z_t}-\right.
  & \left.\frac{\E_x(\mathbbm{1}_A Z_s\eta(X_s))}{\eta(x)}e^{\lambda_0 s}\right|
    =    e^{\lambda_0 s}\left|\frac{\E_x(\mathbbm{1}_A Z_s\eta_{t-s}(X_s))}{\eta_t(x)}-\frac{\E_x(\mathbbm{1}_A
    Z_s\eta(X_s))}{\eta(x)}\right| \notag \\ 
  & \leq e^{\lambda_0 s}\E_x(Z_s\eta_{t-s}(X_s))\frac{|\eta(x)-\eta_t(x)|}{\eta(x)\eta_t(x)}+e^{\lambda_0
    s}\frac{\E_x\left[Z_s|\eta_{t-s}(X_s)-\eta(X_s)|\right]}{\eta(x)} \notag \\
  & \leq\frac{C e^{-\alpha t}}{\underline{\eta}}+\frac{C\bar{\eta}}{\underline{\eta}}e^{-\alpha(t-s)}
    \label{eq:pf-Q-proc}
\end{align}
for some constant $C$, where $\bar\eta=\sup_{t\in I,\,x\in E}\eta_t(x)<+\infty$ and $\underline{\eta}=\inf_{x\in E}\eta(x)>0$. This
proves~\eqref{eq:Q-process-compact} with
\[
  \mathbb{Q}_x(A)=\frac{\E_x(\mathbbm{1}_A Z_s\eta(X_s))}{\eta(x)}e^{\lambda_0 s}.
\]
The rest of Point (i) and Point (ii) can be deduced from Theorem~\ref{thm:eta-compact} exactly as Points~(i) and~(ii) in Theorem~3.1
of~\cite{ChampagnatVillemonais2016}. In addition, the fact that $\nu_Q$ is invariant for $(\widetilde{P}_t)_{t\in I}$ is
straightforward, and its uniqueness will follow from~\eqref{eq:Q-proc-compact}.

So it only remains to prove~\eqref{eq:Q-proc-compact}. We can deduce it from the results of Section~\ref{sec:proofunderABC}: it
follows from~\eqref{eq:QE} that, for all $x,y\in E$, $s\leq t\in I$ and $\phi\in\text{Lip}_1(d_\kappa)$,
\begin{align*}
  \left|\delta_x R_{0,s}^t\phi-\delta_y R_{0,s}^t\phi\right|
  & =e^{\lambda_0 s}\left|\frac{\E_x
    Z_s\phi(X_s)\eta_{t-s}(X_s)}{\eta_t(x)}-\frac{\E_y Z_s\phi(X_s)\eta_{t-s}(X_s)}{\eta_t(y)}\right| \\
  & \leq\frac{C_0}{\beta}\,e^{-\alpha s}\,d_\kappa(x,y).
\end{align*}
Letting $t\rightarrow+\infty$, Theorem~\ref{thm:eta-compact} implies that
\begin{align*}
  \left|\delta_x\widetilde{P}_s\phi-\delta_y\widetilde{P}_s\phi\right|
  & =e^{\lambda_0 s}\left|\frac{\E_x Z_s\phi(X_s)\eta(X_s)}{\eta(x)}-\frac{\E_y Z_s\phi(X_s)\eta(X_s)}{\eta(y)}\right|
    \leq \frac{C_0}{\beta}\,e^{-\alpha s}\,d_\kappa(x,y).
\end{align*}
Integrating this with respect to an optimal coupling between $\mu$ and $\nu$, we deduce that
\begin{equation*}
  \mathcal{W}_{d_\kappa}\left(\mu\widetilde{P}_s,\nu\widetilde{P}_s\right)\leq\frac{C_0}{\beta}\,e^{-\alpha s}\,\mathcal{W}_{d_\kappa}(\mu,\nu).
\end{equation*}
Using the fact that $\kappa^{-1} d_\kappa\leq d\leq (\kappa^{-1}\vee\bar{d}) d_\kappa$, Theorem~\ref{thm:Q-process-compact} follows.

\subsubsection{Proof of Theorem \ref{thm:QED}}
\label{sec:proof-QED}

We give the proof in the case $I=[0,+\infty)$. The proof is similar in the case where $I=\Z_+$. Let $\phi\in\text{Lip}_1(d)$ and set
$\bar{\phi}=\phi-\nu_Q(\phi)$. Then, using the notations of Section~\ref{sec:proof-eta},
  \[
    \left|\mu^x_t(\phi)-\nu_Q(\phi)\right|=\frac{1}{t}\left|\int_0^t\frac{\E_x[Z_t\bar{\phi}(X_s)]}{\E_x
        Z_t}ds\right|=\frac{1}{t}\left|\int_0^t\frac{\E_x[Z_s\bar{\phi}(X_s)\eta_{t-s}(X_s)]}{\eta_t(x)}e^{\lambda_0 s}ds\right|.
  \]
  Now, we deduce from~\eqref{eq:pf-Q-proc} with $A=E$ that
  \[
   e^{\lambda_0 s} \left|\frac{\E_x[Z_s\bar{\phi}(X_s)\eta_{t-s}(X_s)]}{\eta_t(x)}-\frac{\E_x[Z_s\bar{\phi}(X_s)\eta(X_s)]}{\eta(x)}\right|
   \leq \|\bar{\phi}\|_\infty\left(\frac{C e^{-\alpha t}}{\underline{\eta}}+\frac{C\bar{\eta}}{\underline{\eta}}e^{-\alpha(t-s)}\right)
   \leq C'e^{-\alpha(t-s)},
  \]
  for some constant $C'$, where
  we used the fact $\|\bar{\phi}\|_\infty\leq \bar{d}$ since $\bar{\phi}$ is 1-Lipschitz and $\nu_Q(\bar{\phi})=0$.

  Therefore,
  \begin{align*}
    \left|\mu^x_t(\phi)-\nu_Q(\phi)\right|
    & \leq \frac{C'}{t}\int_0^t e^{-\alpha(t-s)}ds+\frac{1}{t}\left|\int_0^t e^{\lambda_0
      s}\frac{\E_x[Z_s\bar{\phi}(X_s)\eta(X_s)]}{\eta(x)}ds\right| \\
    & \leq\frac{C'}{\alpha t}+\frac{1}{t}\int_0^t\left|\E_x^\Q \phi(X_s)-\nu_Q(\phi)\right|ds, 
  \end{align*}
  where $\E^\Q_x$ is the expectation with respect to $\Q_x$. Now, it follows from Theorem~\ref{thm:Q-process-compact}~(iii) that
  \[
    \left|\E_x^\Q \phi(X_s)-\nu_Q(\phi)\right|\leq\mathcal{W}_d\left(\Q_x(X_s\in\cdot),\nu_Q\right)\leq C'' e^{-\alpha s}
  \]
  for some constant $C''$. Hence,
  \[
    \left|\mu^x_t(\phi)-\nu_Q(\phi)\right|\leq \left(\frac{C'}{\alpha}+\frac{C''}{\alpha}\right)\frac{1}{t}
  \]
  and Theorem~\ref{thm:QED} is proved.

\section{Applications}
\label{sec:applications}

\subsection{ General criteria relating the coupling rate of the non-penalized process and the killing rate}
\label{sec:coupling-faster-killing}

 Our goal in this section is to provide a general criterion to deal with convergence of conditional distributions in the
  case where there exists a coupling for the non-penalized process leading to its exponential convergence in Wasserstein distance.

\begin{prop}
  \label{prop:CFTK} ~
  
  \begin{description}
  \item[\textmd{(i)}]  Assume that there exist four constants $C,c>0$ and $\gamma,\kappa>0$ such that
  \begin{gather}
  	\label{eq:subkillingrate}
  	\mathbb E_x(Z_t)\geq c\,e^{-\kappa t},\ \forall t\geq 0,\ \forall x\in E, \\
    \label{assumptionCFTK3}
    \mathbb E_{(x,y)}\left[d(X_t,Y_t)\right]\leq  C\,e^{-\gamma t}\,d(x,y),\ \forall t\geq 0,\ \forall x,y\in E,
  \end{gather}
  for some Markov coupling $\mathbb P_{(x,y)}$ of $\mathbb P_x$ and $\mathbb P_y$ and
  \begin{align}
    \label{ass-gamma}
    \gamma>\kappa-\inf_E\rho.
  \end{align}
  Then Assumption~(A) holds true with $C_A=C/c$ and
  $\gamma_A=\gamma-\kappa+\inf_E\rho$.
\item[\textmd{(ii)}] Assume that there exists a Markovian coupling $(\mathbb P_{(x,y)})_{x,y\in E}$ of the non-penalized process $X$
  with extended generator $L^c$ such that $d$ belongs to its domain and
  \begin{align}
    \label{assumptionCFTK2}
    \sigma:=\sup_{x\neq y} \frac{L^c d(x,y)}{d(x,y)}-\rho(x)+\|\rho\|_\infty<0.
  \end{align}
  Then Assumption~(A) holds true with $C_A=1$ and
  $\gamma_A=-\sigma$ 
  \end{description}
\end{prop}


\begin{proof}[Proof of Proposition~\ref{prop:CFTK}]
  Let us first prove (i).  The fact that $ G_t^X$ is smaller than $\frac1c\,e^{t\,(\kappa-\inf_E\rho)}$ almost surely entails that (A) holds true
  with $C_A=C/c$ and $\gamma_A=\gamma-\kappa+\inf_E\rho$ and the coupling $\mathbb P_{(x,y)}$.

  To prove~(ii), we first observe that $\sigma<0$ and $d$ bounded imply that $L^c d$ is also bounded on $E^2$. Therefore, the process
  \[
    M_t:=d(X_t,Y_t)-d(x,y)-\int_0^t L^c d(X_s.Y_s)ds
  \]
  is a $L^\infty$ martingale, hence $L^2$. It\^o's formula then entails that, for all $0<s<t$,
  \begin{align*}
    e^{-\int_0^t \rho(X_u)du}d(X_t,Y_t) & =
    e^{-\int_0^s \rho(X_u)du}d(X_s,Y_s)+\int_s^t
    e^{-\int_0^r \rho(X_u)du} dM_r \\ & +\int_s^t e^{-\int_0^r \rho(X_u)du}\left[-\rho(X_r)d(X_r,Y_r)+L^c d(X_r,Y_r)\right]dr
  \end{align*}
  where $\int_0^t e^{-\int_0^r \rho(X_u)du} dM_r$ is a martingale. Hence
  \begin{align*}
    \frac{\partial }{\partial t}\mathbb E_{(x,y)}\left[G_t^X d(X_t,Y_t)\right]\leq
    \frac{\mathbb{E}_{(x,y)}\left[e^{-\int_0^t\rho(X_s)\,\mathrm ds}\left(-\rho(X_t)d(X_t,Y_t)+L^cd(X_t,Y_t)+\|\rho\|_\infty
    d(X_t,Y_t)\right)\right]}{\mathbb E_{x}\left[e^{-\int_0^t\rho(X_s)\,\mathrm ds}\right]}.
  \end{align*}
  Therefore the conclusion follows from~\eqref{assumptionCFTK2} and Gronwall lemma.
\end{proof}

\begin{rem}
  \label{rem:coupling-faster-than-killing}
  We first observe that~\eqref{eq:subkillingrate} always holds true when $\kappa=\|\rho\|_\infty$, in which case the condition in (i) reduces to $\gamma>\text{osc}(\rho)$.

  Let us now compare (i) and (ii). If~\eqref{assumptionCFTK2} holds true, then $\sup_{x\neq y} \frac{L^c d(x,y)}{d(x,y)}<0$ and
  therefore~\eqref{assumptionCFTK3} holds true with $\gamma=-\sup_{x\neq y} \frac{L^c d(x,y)}{d(x,y)}$ and $C=1$.
  However,~\eqref{assumptionCFTK2} does not imply~\eqref{ass-gamma} as shown by the example of the Markov chain on
  $\{1,2,3\}$ with transition rate matrix
  \[
    \begin{pmatrix}
      -2K & K & K \\
      \varepsilon & -2\varepsilon & \varepsilon \\
      \varepsilon & \varepsilon & -2\varepsilon
    \end{pmatrix}
  \]
  and killing rate $\rho=\1_{\{2,3\}}$. Straightforward computations imply that, provided $K\geq 1$, $\sigma<0$ for the independent
  coupling, but $\gamma\leq 4\varepsilon$ for any coupling and $\kappa-\inf\rho\geq 2/3$ for any $K\geq\varepsilon$.
    
  Conversely, in general~\eqref{assumptionCFTK3} does not imply that $\sup_{x\neq y} \frac{L^c d(x,y)}{d(x,y)}<0$ and therefore,
  taking $\rho\equiv 0$ or close to 0 would give a model that satisfies~\eqref{eq:subkillingrate},~\eqref{assumptionCFTK3}
  and~\eqref{ass-gamma} but contradicts~\eqref{assumptionCFTK2}. This shows that Points~(i) and~(ii) of Proposition~\ref{prop:CFTK}
  cover distinct cases.
 
\end{rem}

 Let us now discuss related results from the literature. Namely,
in~\cite{CloezThai2016,Villemonais2020,JournelMonmarche2022}, the authors consider an interacting particle system built on an
underlying absorbed Markov process, and they prove, under a small Lipschitz norm assumption for the absorbing rate, the exponential
contraction of the law of the interacting particle system, uniformly in the number of particle. Since this particle system converges
toward the conditional distribution of the underlying absorbed Markov process when the number of particles goes to infinity, this
entails the exponential contraction in Wasserstein distance of the conditional distribution of the underlying Markov process. In
particular, in the continuous time setting, the sharpest explicit assumption in the above references (see
\cite[Remark~2.4]{CloezThai2016} for the trivial distance, and \cite[Remark~7]{Villemonais2020} with $\theta_P=0$ for the general
case) is that the distance $d$ is bounded (say by $\bar{d}$), that $X$ is a non-explosive pure jump Markov process, that $\rho$ is
Lipschitz and that there exists a Markovian coupling $(\mathbb P_{(x,y)})_{x,y\in E}$ of $X$ with extended infinitesimal generator
$L^c$ and such that
\begin{align}
\label{assumptionCFTK1}
\sigma_1:=\sup_{x\neq y} \frac{L^c d(x,y)}{d(x,y)}+\frac{|\rho(x)-\rho(y)|}{d(x,y)}\bar{d}-\rho(x)\vee \rho(y)+\|\rho\|_\infty<0.
\end{align}
Note that we have $\sigma_1\leq \sigma$ and, importantly, there is no condition on the Lipschitz norm of $\rho$
in~\eqref{assumptionCFTK2}. Therefore,~\eqref{assumptionCFTK2} is less stringent than~\eqref{assumptionCFTK1}  and
Proposition~\ref{prop:CFTK} improves the results of~\cite {Villemonais2020}.

\subsection{Almost sure contracting coupling}
\label{sec:almost-sure-contraction}
In this section we assume that for all $x, y \in E$, there exists a Markovian coupling  measure between $\mathbb{P}_x$ and $\mathbb{P}_y$ such that for all $t \in I$,
\begin{equation}
\label{eq:ascontraction}
d(X_t, Y_t) \leq C_0 e^{ - \gamma t} d(x,y) \quad \mathbb{P}_{(x,y)} - \text{almost surely}.
\end{equation}
In this situation, Assumption (A) is trivially satisfied and therefore the conclusions of Theorem \ref{thm:main-compact} hold
true. Importantly, no assumption on $\gamma-\text{osc}(\rho)$ is required (compare with Proposition~\ref{prop:CFTK}). To obtain a
better rate of convergence $\alpha$ in Equations \eqref{eq:main-compact} and \eqref{eq:hyp-eta} (see Remark~\ref{rem:explicit}), we
prove below conditions (A'), (B) and (C) with explicit constants

\begin{prop}
	\label{prop:almostsure}
  Assumptions (A'), (B) and (C) hold with $\gamma(t)=C_0 e^{ - \gamma t}$ , 
  \begin{equation}
    \label{eq:def-C_B-a.s.-contr}
    C_B=\frac{C_0e^{\text{osc}(\rho)}\|\rho\|_{\text{Lip}(d)}}{1-e^{-\gamma}}
    \exp\Big(\frac{C_0e^{\text{osc}(\rho)}\|\rho\|_{\text{Lip}(d)}\bar{d}}{1-e^{-\gamma}}\Big)
  \end{equation}
  in the discrete time case (where we recall the notation $p(x)=e^{-\rho(x)}$),  
  \[C_B = \|\rho\|_{\text{Lip}(d)} \frac{C_0}{\gamma} \exp( \|\rho\|_{\text{Lip}(d)} \frac{C_0}{\gamma} \bar d)\]
  in the continuous time case, and in both cases,
  \[C_C = (1 + C_B \bar d)^2.\]
\end{prop}

\begin{proof}
We only write the proof in the discrete time case, the continuous time case being similar.
(A') is a straightforward consequence of \eqref{eq:ascontraction} since $\E_{(x,y)}(G_t^X)= 1$. We now prove (B). Our assumptions entail that, for all $x,y\in E$,
   \begin{align*}
   |p(x)-p(y)|\leq (p(x)\vee p(y))\|\rho\|_{\text{Lip}(d)} d(x,y)\leq e^{\text{osc}(\rho)}\|\rho\|_{\text{Lip}(d)} d(x,y)p(y).
   \end{align*}
    We deduce that, $\mathbb P_{(x,y)}$-almost surely for all $x,y\in E$, for all $n\geq 1$,
   \begin{align*}
   p(X_0)\cdots p(X_{n-1}) &\leq (1+e^{\text{osc}(\rho)}\|\rho\|_{\text{Lip}(d)}d(X_0,Y_0)) p(Y_0) \\ & \qquad\qquad\cdots
                             (1+e^{\text{osc}(\rho)}\|\rho\|_{\text{Lip}(d)} d(X_{n-1},Y_{n-1})) p(Y_{n-1}).
   \end{align*}
   Using~\eqref{eq:ascontraction}, this implies
   \begin{align*}
   p(X_0)\cdots p(X_{n-1})
   &\leq (1+C_0 e^{\text{osc}(\rho)}\|\rho\|_{\text{Lip}(d)} d(X_0,Y_0)) \\ & \qquad\qquad \cdots (1+C_0 e^{\text{osc}(\rho)}\|\rho\|_{\text{Lip}(d)}
     e^{-\gamma (n-1)} d(X_0,Y_0))\, p(Y_0)\cdots p(Y_{n-1})\\
   &\leq (1+C_B d(X_0,Y_0))\, p(Y_0)\cdots p(Y_{n-1}),
   \end{align*}
   where $C_B$ is defined in~\eqref{eq:def-C_B-a.s.-contr}. Assuming that $p(Y_0)\ldots p(Y_{n-1})\leq p(X_0)\ldots p(X_{n-1})$
   (otherwise, exchange $X$ and $Y$ in the following inequality),
\begin{multline*}
p(X_0)\cdots p(X_{n-1}) - p(Y_0)\cdots p(Y_{n-1})\leq C_B d(X_0,Y_0)\, p(Y_0)\cdots p(Y_{n-1})\\
\leq C_B d(X_0,Y_0) (p(X_0)\ldots p(X_{n-1}))\wedge(p(Y_0)\ldots p(Y_{n-1})).
\end{multline*}
This entails that, $\mathbb P_{(x,y)}$-almost surely for all $x,y\in E$, we have
\begin{align}
\label{eq:useful1}
| Z_{n}^X-Z_{n}^Y|\leq C_B d(x,y) Z_n^X\wedge Z_n^Y.
\end{align}
Taking the expectation implies (B).

To prove (C), we observe that $d\leq \bar d$ and hence, from~\eqref{eq:useful1}, for all $n\geq 1$ and all $x,y\in E$,
\begin{align*}
Z_n^Y\leq (1 + C_B \bar d)\,Z_n^X,\quad\P_{(x,y)}-as.
\end{align*}
In particular, 
\begin{align*}
{G_n^X=}\frac{Z_n^X}{\E_x\left(Z_n^X\right)}\geq \frac{1}{(1 + C_B \bar d)^2}\frac{Z_n^Y}{\E_y\left(Z_n^Y\right)}{= \frac{1}{(1 + C_B
  \bar d)^2} G_n^Y}.
\end{align*}
{Since $\mathbb{E}_x G_n=1$, this implies }
that (C) holds true with $C_C = (1 + C_B \bar d)^2$.
\end{proof}
Recall that from \eqref{eq:alpha-explicite}, we have
\[
\alpha=\frac{\gamma\log(\frac{2C_C}{2C_C-1})}{\log\left(2 C_0(1+C_B\bar d)[1\vee\frac{2 C_B C_C^2\bar d}{C_C-1}]\right)}.
\]
Set $\tilde \gamma = \frac{C_0 \|\rho\|_{\text{Lip}(d)}\bar{d}}{\gamma}$ and observe that
$C_B\bar{d}=\tilde{\gamma}e^{\tilde{\gamma}}$. In particular, we see that when $C_0\|\rho\|_{\text{Lip}(d)}/\gamma$ is close to 0
(e.g.\ either when $\rho$ is nearly constant or the almost sure convergence in~\eqref{eq:ascontraction} is very fast\footnote{Note
  from~\eqref{eq:ascontraction} that $C_0\geq 1$}, or more generally when the rate of contraction $\gamma$ is much larger than the
variation of the killing rate), $C_B$ is close to 0 and hence $\alpha$ is close to $\gamma \frac{\log(2)}{\log(2C_0)}$. In
particular, if $C_0= 1$, we recover the convergence rate of the non-penalized trajectories. This is not the case when using (A) and
Theorem~\ref{thm:main-compact}.

\begin{exa}[Iterated random functions with penalization]
  \label{ex:IRF}
Let $M$ be a measurable space, and consider a familly $(f_{\theta})_{\theta \in M}$  of  functions defined on $E$ such that
\[
a:=\sup_{\theta} \|f_\theta\|_{\text{Lip}(d)} < 1.
\]
 We construct the Markov chain defined recursively by 
\[
X_{n+1} = f_{\theta_{n+1}}(X_n),
\] 
where $(\theta_n)_{n \geq 1}$ is an i.i.d.\ sequence of random variables with values in $M$. We penalize $X$ by a function $p=e^\rho$
with $\rho\in \text{Lip}(d)$. For all $x,y$ we couple $(X_n)_{n \geq 0}$  and $(Y_n)_{n \geq 0}$ using the same realisation of the sequence $(\theta_n)_{n \geq 1}$. In particular, for all $n \geq 0$,
\[
d(X_{n+1}, Y_{n+1}) = d(f_{\theta_{n+1}}(X_{n+1}),f_{\theta_{n+1}}(Y_{n+1})) \leq a d(X_n, Y_n).
\]
Therefore, \eqref{eq:ascontraction} is satisfied with $\gamma=-\log a$, so Assumptions (A'), (B) and (C) hold true, which
implies  the conclusion of Theorem \ref{thm:main-compact}

The example of the Bernouilli convolution given in the Introduction is the special case where $M = \{-1, +1\}$,
$E = [ -2, 2]$, $f_{\theta} : x \mapsto \frac{x}{2}+\theta$ and $(\theta_n)_{n \geq 1}$ is a sequence of i.i.d.\ Rademacher variables.
\end{exa}

\subsection{Iterated random contracting functions}
\label{sec:random-composition}
Let $M$ be a measurable space, and consider a family $(f_{\theta})_{\theta \in M}$ of Lipschitz functions from $E$ to $E$. We assume
that for all $\theta \in M$, $l_{\theta} := \|f_\theta\|_{\text{Lip}(d)} \leq 1$, but we do not assume as in Example~\ref{ex:IRF}
that $\sup_{\theta\in M}\ell_\theta<1$. We consider the Markov chain defined recursively by
\[
X_{n+1} = f_{\theta_{n+1}}(X_n),
\] 
where $(\theta_n)_{n \geq 1}$ is an i.i.d.\ sequence of random variables with values in $M$. We penalize $X$ by a function
$p=e^{-\rho}$ where $\rho \in \text{Lip}(d)$. We prove the following result:
\begin{prop}
  Assume that there exist two constants $c>0$ and $\kappa>0$ such that 
  \begin{align}
  	\label{eq:subkillingrateIRF}
  	\mathbb E_x(Z_n)\geq c\,e^{-\kappa n},\ \forall n\geq 0,\ \forall x\in E.
  \end{align}
If in addition $q:= \mathbb{P}( l_{\theta_1} = 1) < e^{ -(\kappa - \inf_E \rho )}$, then (A) holds true.
\end{prop}

 Note that this result is much stronger that what one could expect from Proposition~\ref{prop:CFTK} since, in the case
  where $\mathbb{P}( l_{\theta_1} = 1)=0$, (A) holds true whatever the value of $\kappa$. This result is also stronger that one could expect  from Proposition~\ref{prop:almostsure} since  $\mathbb{P}( l_{\theta_1} = 1)=0$ does not imply~\eqref{eq:ascontraction} in general.

 As noticed in Section~\ref{sec:coupling-faster-killing}, condition~\eqref{eq:subkillingrateIRF} is always satisfied with $\kappa = \sup_E \rho$. Another easy way to check~\eqref{eq:subkillingrateIRF} is as follows. Assume that there exists $\alpha \in (0,1] $ such that, for all $x \in E$, $\mathbb{E}_x(p(X_1)) \geq \alpha$. Then, the Markov property implies that $\mathbb{E}_x( Z_n) \geq \alpha \mathbb{E}_x(Z_{n-1})$, and thus~\eqref{eq:subkillingrateIRF} holds true with $c = \alpha^{-1} e^{ - \sup_E \rho}$ and $\kappa =- \log(\alpha)$.

For all $x,y$ we couple $(X_n)_{n \geq 0}$  and $(Y_n)_{n \geq 0}$ using the same realisation of the sequence $(\theta_n)_{n \geq 1}$. In particular, for all $n \geq 0$,
\[
d(X_{n+1}, Y_{n+1}) = d(f_{\theta_{n+1}}(X_{n+1}),f_{\theta_{n+1}}(Y_{n+1})) \leq l_{\theta_{n+1}} d(X_n, Y_n).
\]
Hence, to prove (A), it is sufficient to prove that for some $a \in (0,1)$ and $C >0$, for all $n \geq 1$,
\[
\E \left[ \left( \prod_{k=1}^n l_{\theta_k} \right) p(X_0) \ldots p(X_{n-1}) \right] \leq C a^n \E\left[ p(X_0) \ldots p(X_{n-1}) \right]
\] 
We let $\nu_n = l_{\theta_n}$. Note that $(\nu_n)_{n \geq 0}$ is an i.i.d.\ sequence of random variables  with values in $[0,1]$. Let $\varepsilon, \delta > 0$, and set 
\[
A_n(\varepsilon) = \left\{ \frac{1}{n} \sum_{k=1}^n \nu_k \leq 1 - \varepsilon \right\}, \quad B_n(\delta) = \{ 1 \leq k \leq n : \: \nu_k \leq \delta\}
\]
On the event $A_n(\varepsilon)$, for all $ 1 - \varepsilon < \delta < 1$, one can prove that $|B_n(\delta)| \geq n(1 - \frac{1- \varepsilon}{\delta})$. Therefore, since $\nu_k \leq 1$ for all $k \geq 1$, on  $A_n(\varepsilon)$ one has
\begin{equation}
\label{eq:onA_n}
\prod_{k=1}^n \nu_k \leq \prod_{k \in B_n(\delta)} \nu_k \leq \delta^{n\left(1 - \frac{1-\varepsilon}{\delta}\right)}.
\end{equation}
Moreover, Cramer's Theorem implies that for all $n \geq 1$ and all $\varepsilon \in [0,1]$, 
\[
\P \left( A_n(\varepsilon)^{\mathrm{c}} \right) \leq e^{ - n \Lambda^*(1-\varepsilon)}
\]
where $\Lambda^*$ is the Fenchel-Legendre transform of the {cumulant} generating function of $\nu_1$, defined for all $x \in \R$ as
\[
\Lambda^*(x) = \sup_{t \geq 0} \{ tx - \ln \E ( e^{ t \nu_1})  \}.
\]
We prove the following lemma.
\begin{lem}
\label{lem:FL1}
We assume that $q:= \P( \nu_1 =1) \in [0,1)$. Then
\[
\Lambda^*(1) \geq \ln(1/q) \in (0, +\infty].
\]
In particular, 
\begin{equation}
\label{eq:FLeps}
\liminf_{ \varepsilon \to 0} \Lambda^*(1 - \varepsilon) \geq \ln(1/q).
\end{equation}
\end{lem}
\begin{proof}
  Note that $\Lambda^*(1) =\sup_{t\geq 0}h(t)$, where $h(t) = t - \ln \E( e^{t \nu_1})$ for all $t \geq 0$. Since $\nu_1$ is bounded,
  we can differentiate to get $h'(t) = 1 - \frac{\E( \nu_1 e^{t \nu_1})}{\E ( e^{t \nu_1}))}$. Since $q < 1$,
  $\E( \nu_1 e^{t \nu_1}) < \E( e^{t \nu_1})$, and thus $h'(t) > 0$ for all $t \geq 0$. Therefore,
  $\Lambda^*(1) = \lim_{t \to \infty} h(t)$. Furthermore, if $F$ denotes the cumulative distribution function of $\nu_1$, we have for
  all $\varepsilon \in (0,1)$,
  \begin{align*}
    h(t) & \geq t - \ln \left( e^t (1 - F(1 - \varepsilon)) + e^{ (1- \varepsilon) t} F(1- \varepsilon) \right)\\
         & = {-}\ln \left( (1 - F(1 - \varepsilon)) + e^{ - \varepsilon t} F(1- \varepsilon) \right)
  \end{align*}
  This entails that for all $\varepsilon \in (0,1)$, $\lim_{t \to \infty} h(t) \geq {-}\ln( (1 - F(1 - \varepsilon))$. Letting
  $\varepsilon \to 0$ yields $\Lambda^*(1) \geq \ln (1/q)$. Equation \eqref{eq:FLeps} follows from the fact that $\Lambda^*$ is lower
  semicontinuous.
\end{proof}
We can now prove (A). We decompose: for all $\varepsilon \in (0,1)$,

\begin{align*}
\E \left[ \left( \prod_{k=1}^n \nu_k \right) p(X_0) \ldots p(X_{n-1}) \right] = \E & \left[\1_{A_n(\varepsilon)} \left( \prod_{k=1}^n \nu_k \right) p(X_0) \ldots p(X_{n-1}) \right]
\\
&  + \E \left[\1_{A^c_n(\varepsilon)} \left( \prod_{k=1}^n \nu_k \right) p(X_0) \ldots p(X_{n-1}) \right]
\end{align*}
By Equation \eqref{eq:onA_n}, the first term of the right hand side in bounded above by 
\[
\delta^{n\left(1 - \frac{1-\varepsilon}{\delta}\right)} \E\left[ p(X_0) \ldots p(X_{n-1}) \right].
\]
The second term is bounded above by 
\[
e^{-n\inf_{y\in E}\rho(y)}\P \left( A_n(\varepsilon)^{\mathrm{c}} \right) \leq e^{ - n (\Lambda^*(1-\varepsilon)+\inf_{y\in E}\rho(y))}.
\]
By Lemma \ref{lem:FL1} and the fact that $q < e^{-(\kappa - \inf_E \rho )}$, one can choose $\varepsilon > 0$ such that  $\chi  := \Lambda^*(1-\varepsilon) - (\kappa - \inf_E \rho ) > 0$. Therefore,
\[
 e^{-n\inf_{y\in E}\rho(y)} \P \left( A_n(\varepsilon)^{\mathrm{c}} \right) \leq e^{ - n (\chi+\kappa)}.
\]
Since we assumed in~\eqref{eq:subkillingrateIRF} that $\E \left[  p(X_0) \ldots p(X_{n-1}) \right] \geq c e^{-n \kappa}$, we conclude that
\[
\E \left[ \left( \prod_{k=1}^n \nu_k \right) p(X_0) \ldots p(X_{n-1}) \right]  \leq \left[ \delta^{n\left(1 - \frac{1-\varepsilon}{\delta}\right)} +  \frac{1}{c} e^{ - n \chi} \right] \E \left[  p(X_0) \ldots p(X_{n-1}) \right],
\]
which entails (A).

\subsection{Switched dynamical systems}
\label{sec:switched}

We have seen in the last section how to check Condition (A) for processes satisfying a property of almost sure exponential
coupling. In many situations, for example for PDMP, a property of almost sure coupling occurs only after a stopping time.

\subsubsection{General result}
\label{sec:general-result-PDMP}

\newcommand{\ES}{{\mathcal E}}
\newcommand{\cce}{{c_\ES}}
\newcommand{\tce}{{t_\ES}}
We consider a switched dynamical system $(X_t,I_t)_{t\in\mathbb{R}_+}$, where the environment $I_t$ is an irreducible continuous time Markov
chain on a finite state space $\ES$ and $X_t\in\mathbb{R}^k$  is solution to 
\[
\dot{X}_t=F_{I_t}(X_t),
\]
where, for all $i\in \ES$, $F_i$ is a vector field in $\mathbb{R}^k$
such that 
\begin{equation}
  \label{eq:monotonie}
  \langle F_i(x) - F_i(y), x-y \rangle \leq - \gamma \| x -  y \|^2
\end{equation}
for some constant $\gamma>0$ which does not depend on $i\in\ES$. 

It is not hard to check that, for any $R$ large enough, the closed ball $\|x\| \leq R$ is positively invariant for the flows generated by the
$F_i$. Hence, we can restrict the state space of the  process to that ball, say $M$, times the space of the environment $\ES$. We
let $d$ the distance on $E = M \times \mathcal E$ be defined as 
\begin{equation}
\label{eq:distance-PDMP}
d((x,i),(y,j)) = \1_{i \neq j} + \1_{i = j}\frac{\|x-y\|}{2R},\ \forall x,y\in M,\ i,j\in \ES,
\end{equation}
where $2R$ is the diameter of $M$. We consider the penalization
\begin{equation}
\label{eq:hyp-killing-a.s.PDMP}
Z_t=\exp\left(-\int_0^t\rho(X_s,I_s)\,ds\right)  .
\end{equation}
with $\rho: E \to \mathbb R_+$ a Lipschitz function. 

We consider the usual independent Markov coupling: for all $x,y\in M$ and $i,j\in\ES$, define $\P_{(x,i),(y,j)}$ as the joint law
of $(X_t,I_t,Y_t,J_t)_{t\geq 0}$, where $(X_t,I_t)_{t\geq 0}$ is distributed as $\P_{(x,i)}$, $(Y'_t,J'_t)_{t\geq 0}$ is distributed
as $\P_{(y,j)}$ and independent of $(X_t,I_t)_{t\geq 0}$, and define
\[
T_0:=\inf\left\{t\geq 0: I_t=J'_t\right\}
\]
and, for all $t\geq 0$,
\begin{equation}
\label{eq:indep-coupling}
J_t=J'_t\1_{T_0>t}+I_t\1_{T_0\leq t}\quad\text{and}\quad  Y_t=Y'_t\1_{T_0>t}+\left(Y'_{T_0}+\int_{T_0}^t F_{I_s}(Y_s)ds\right)\1_{T_0\leq t}.  
\end{equation}
In particular, $I_t=J_t$ for all $t\geq T_0$ almost surely.

We can now state the main result of this section.

\begin{prop}
    \label{prop:PDMP}
    Condition (A) holds true {with the above coupling measure}.
\end{prop}

\begin{rem}
  Since $R$ can be chosen arbitrarily large, Proposition~\ref{prop:PDMP} shows that the quasi-stationary distribution of
  Theorem~\ref{thm:main-compact} has compact support and that all compactly supported initial distribution are in its domain of
  attraction. In particular, it is the only compactly supported quasi-stationary distributions.
\end{rem}

\begin{proof}[Proof of Proposition~\ref{prop:PDMP}]
  To shorten notation, we write $u = (x,i)$, $v=(y,j)$ for generic points in $E$ and $U_t = (X_t, I_t)$, $V_t = (Y_t, J_t)$ for the
  processes starting from these points. We first consider the case $i=j$. In this situation, we have,
  $\mathbb P_{(x,i),(y,i)}$-almost surely, for all $t \geq 0$, $I_t = J_t$ and thus
  \begin{align}
    \label{eq:youpi}
    d(U_t,V_t) = \frac{\| X_t - Y_t\|}{2R} \leq e^{-\gamma t}\,\frac{\|x-y\|}{2R} = e^{-\gamma t} d(u,v)
  \end{align}
  and hence (A) holds true in that case. Now, we consider the case where $i \neq j$. We write
  \begin{equation}
    \label{eq:decoupalage-subtil}
    \E_{(u,v)} ( G_t^U d(U_t, V_t) ) \leq \E_{(u,v)} ( G_t^U d(U_t, V_t) \1_{T_0 \leq t /2}) + \E_{(u,v)} ( G_t^U \1_{T_0> t /2}).
  \end{equation}
  Note that on the event $\{T_0 \leq t /2\}$, 
  \begin{align*}
    d(U_t,V_t) & = \frac{\|X_t - Y_t\|}{2R} \leq e^{ - \gamma ( t - T_0)} \frac{\|X_{T_0} - Y_{T_0}\|}{2R} \leq e^{ - \gamma t/2} = e^{ - \gamma t/2} d(u,v).
  \end{align*}
  Hence, the first term in the right-hand side of \eqref{eq:decoupalage-subtil} is bounded by $e^{ - \gamma t/2} d(u,v)$. We now show
  that the second term is bounded by $e^{ - \delta t/2}$ for some $\delta > 0$. To do so, we prove by induction that, for all
  {$n\geq 1$ and all $t\geq n$,
  \begin{equation}
    \label{eq:recurrence-a.s.PDMP}
    \mathbb E_{(u,v)}[\mathbbm{1}_{T_0> n}Z_t^X]\leq (1-c)^n\,\mathbb E_{(x,i)} Z_t    
  \end{equation}}
  for some constant $c\in(0,1)$. Using the independence between $(X_t,I_t)_{t\geq 0}$ and $(Y'_t,J'_t)_{t\geq 0}$, we have {
  \begin{align*}
    \E_{(u,v)}[ Z_t^X \1_{T_0 > 1} ]
    & = \sum_\ell \E_{(u,v)} [ Z_t^X \1_{I_{1} = \ell} \1_{ J_{1} \neq \ell} ] \\
    & \leq \sum_\ell \E_{(u,v)} [ Z_t^X \1_{I_{1} = \ell} \1_{ J'_{1} \neq \ell} ] \\
    & = \sum_\ell \E_{(u,v)} [ Z_{t}^X \1_{I_{1} = \ell}]\mathbb{P}_j\left( J'_{1} \neq \ell \right).
  \end{align*}
  Let $ c = \inf_{j,\ell} \mathbb{P}_j( J'_{1} = \ell) > 0$,} then
  \[
    \E_{(u,v)}[ Z_t^X \1_{T_0 > t_1} ] 
    \leq (1-c)\E_{(u,v)} Z_t^X,
  \]
  which proves~\eqref{eq:recurrence-a.s.PDMP} for {$n=1$}. Now let us assume~\eqref{eq:recurrence-a.s.PDMP} and prove it
  for {$n+1$:}
  applying the Markov property at time {$n$ twice, we deduce that, for all $t\geq n+1$,
  \begin{align*}
    \mathbb E_{(u,v)}[\mathbbm{1}_{T_0> n+1}Z_t^X]
    & =\mathbb E_{(u,v)}[\mathbbm{1}_{T_0> n}Z_{n}^X\mathbb
      E_{(X_{n},I_n),(Y_{n},J_n)}(\mathbbm{1}_{T_0>1}Z_{t-n}^X)] \\
    & =(1-c)\,\mathbb E_{(u,v)}[\mathbbm{1}_{T_0> n}Z_{t}^X]\leq (1-c)^{n+1}\,\mathbb E_{(x,i)} Z_t.
  \end{align*}}
  Hence~\eqref{eq:recurrence-a.s.PDMP} is proved and the proof of Proposition \ref{prop:PDMP} is completed.
\end{proof}

\begin{rem} Our assumption that $\ES$ is finite can be easily relaxed to the situation where
    $\ES$ is a Polish space, assuming that there exist constants $\cce,\tce>0$ such that, for all $i,j\in\mathcal E$,
    \begin{align}
    \mathbb{P}_i(I_{\tce}\in\cdot)\geq \cce\,\mathbb{P}_j(J_{\tce}\in\cdot)
    \end{align}
    and that there exists a coupling $\mathbb P_{i,j}$ of $\mathbb P_i$ and $\mathbb P_j$ such that $T_0=\inf\{t\geq 0,\ I_t=J_t\}$ satisfies
    \begin{align}
    \inf_{i,j} \mathbb P_{i,j}(T_0\leq \tce)\geq \cce\text{ and }\inf_{i,j} \mathbb P_{i,j}(I_t=J_t\ \forall t\geq T_0)=1.
    \end{align}
\end{rem}%

\begin{exa}
  We consider the particular case where $k=2$, $F^1(x)=Ax$ and $F^2(x)=A(x-a)$, for some $a\in\mathbb{R}^2$ and some matrix $A$ whose
  eigenvalues have negative real parts. This example was studied in~\cite{BMZIHP}, where the authors check that~\eqref{eq:monotonie}
  is satisfied. Hence Proposition~\ref{prop:PDMP} holds true. In the particular case where $a$ is an eigenvector of $A$, it is easy
  to check that the line $a\mathbb{R}$ (resp.\ the set $\mathbb{R}^2\setminus (a\mathbb{R})$) is invariant for the flows of $F^1$ and
  $F^2$, and hence $a\mathbb{R}\times\ES$ (resp.\ $(\mathbb{R}^2\times\ES)\setminus (a \mathbb{R}\times\ES)$) is invariant for the
  PDMP. In particular, taking an initial condition $x\in a\mathbb{R}$ and $y\not\in a\mathbb{R}$, the total variation distance
  between the conditional laws of the two processes is constant equal to $1$. In contrast, our results imply convergence in
  Wasserstein distance.
\end{exa}


\subsubsection{ Penalization might change the stability of an equlibrium}

 In this section, we give an example where the non-penalized process admits several stationary distributions, one of which being
unstable. Our purpose is to illustrate that adding appropriate penalization may reverse the stability of this equlibirum. More
specifically, we consider a switched dynamical system from Hurth and Kuehn in \cite{HK20} which may admit up to three ergodic
stationary distributions.
  
Consider the process $X_t \in \mathbb{R}$ solution to
\[
\dot{X_t} = F_{I_t}(X_t),
\]
where $I_t$ is a Markov chain on $\mathcal{E}=\{-,+\}$ with transition rates matrix
\[
Q = \begin{pmatrix}
- 1 & 1\\
1 & -1
\end{pmatrix}
\]
and 
\[
F_-(x) = p_- x - x^3, \quad F_+(x) = p_+ x - x^3,
\]
for some $p_+ >  0 > p_-$. The state space of the process is $E = [ - \sqrt{p_+}, \sqrt{p_+}] \times \mathcal{E}$. The following dichotomy is proven in \cite{HK20}, using linearisation results of \cite{BS19}:
\begin{prop} (Theorem~3.1 in~\cite{HK20})
  \label{prop:HK}
\begin{description}
\item[\textmd{(i)}] If $p_+ \leq - p_-$, then for all $(x,i) \in E$,  $X_t$  converges to $0$ $\P_{(x,i)}$-almost surely.
\item[\textmd{(ii)}] If $p_+ >  - p_-$, then $(X_t, I_t)_{ t \geq 0}$ admits 3 ergodic measures: $\delta_0 \otimes (\frac{1}{2}\delta_- +\frac{1}{2}\delta_+) $, $\Pi_-$ and $\Pi_+$, where $\Pi_-( [-\sqrt{p_+}, 0) \times \mathcal{E} ) = 1$ and $\Pi_+( (0, \sqrt{p_+}] \times \mathcal{E} ) = 1$. Moreover, if $x > 0$, then $(X_t, I_t)$ converges in law to $\Pi_+$ while if $x < 0$, then $(X_t, I_t)$ converges in law to $\Pi_-$.
\end{description}
\end{prop}

Now we penalize $(X,I)$ by the function $\rho : E \to \mathbb{R}_+$ defined by 
\[
\rho(x,-) = 0, \quad \rho(x,+) = r  > 0.
\]
We write $\rho(-)$ (resp.\ $\rho(+)$) for $\rho(x,-)$ (resp.\ $\rho(x,-)$) for simplicity. We show that, provided $r$ is large enough
and regardless of the sign of $p_+ + p_-$, the penalized process admits $\delta_0 \otimes \mu$ as unique quasi-stationary
distribution, to which conditional distributions converge for the Wasserstein distance. Here $\mu$ denotes the unique
quasi-stationary distribution of $I_t$ penalized by $\rho$.

\begin{prop}
\label{prop:rlarge(A)}
If $r$ is large enough, then (A) is satisfied for the distance $d$ defined by \eqref{eq:distance-PDMP} with $M = [ - \sqrt{p_+}, \sqrt{p_+}]$ and $R = 2\sqrt{p_+}$.
\end{prop}
\begin{rem}
  \label{rem:last-example}
   In all the previous examples, we were able to prove (A) by using contraction properties of the non-penalized
    trajectories. In particular, in the previous examples, the non-penalized processes all converge uniformly in $L^1$-Wasserstein
    distance to a unique stationary distribution. The example of this section provides a situation where (A) holds true, without
    uniform contraction for the non-penalized process. More precisely, if $p_+ > - p_-$, Proposition~\ref{prop:HK} implies that there
    is no contraction in Wasserstein distance for $(X_t, I_t)$ on $E$. If we restrict the process to
    $(0,\sqrt{p_+}] \times \mathcal{E}$, there may be convergence of the non-penalized process in Wasserstein distance to $\Pi_+$,
    however non-uniform because the process takes arbitrary large time to leave a neighborhood of $\{0\}\times\mathcal{E}$ when $X_0$ is
    close to 0.
\end{rem}
\noindent To prove Proposition~\ref{prop:rlarge(A)}, we need the following Lemma, which is a consequence of results in
\cite{BardetGuerinEtAl2010}.
\begin{lem}
\label{lem:r-large-conctrat}
For $r$ large enough, there exist $C, \gamma > 0$ such that, for all $i \in \mathcal{E}$ and $t \geq 0$,
\begin{equation}
\E_i\left( e^{  \int_0^t p_{I_s} ds } Z_t \right) \leq C e^{ - \gamma t} \E_i \left( Z_t\right)
\end{equation}
\end{lem}

\begin{proof}
For some $(q_+, q_-) \in \mathbb{R}^2$, let $\theta(q_+,q_-)$ be the greatest eigenvalue of $Q - \mathrm{Diag}(q_+,q_-)$, given by 
\[
\theta(q_+,q_-) = - 1 + \frac{q_+ + q_-}{2} + \frac{\sqrt{(q_+ - q_-)^2+4}}{2}.
\]
By Proposition 4.1 in \cite{BardetGuerinEtAl2010} (which is essentially a consequence of Perron-Frobenius Theorem), there exist $0 < C_1(q_+, q_-) < C_2(q_+,q_-)  < + \infty$ such that, for all $i \in \mathcal{E}$ and $t \geq 0$, 
\begin{equation}
\label{eq:BGM}
C_1(q_+, q_-) e^{ \theta(q_+,q_-) t} \leq \E_i \left( e^{ \int_0^t q_{I_s} ds} \right) \leq C_2(q_+, q_-) e^{ \theta(q_+,q_-) t}.
\end{equation}
This entails 
\[
\E_i\left( e^{  \int_0^t p_{I_s} ds } Z_t \right) = \E_i\left( e^{  \int_0^t p_{I_s} ds } e^{  - \int_0^t \rho(I_s) ds } \right) \leq C_2(p_+ - r, p_-) e^{ \theta(p_+ - r, p_-)t},
\]
and
\[
\E_i(Z_t) \geq C_1(-r,0) e^{ \theta( - r, 0) t}.
\]
In other words, 
\[
\E_i\left( e^{  \int_0^t p_{I_s} ds } Z_t \right)  \leq C e^{ - \gamma t} \E_i(Z_t),
\]
with $C = \frac{C_2(p_+ - r, p_-)}{C_1(-r,0)}$ and $\gamma = \theta(-r, 0) - \theta( p_+ - r, p_-)$. Thus, the expected result is proven if we show that $\gamma > 0$ for $r$ large enough. It is easily checked that 
\[
\lim_{ r \to \infty} \theta(-r, 0) - \theta( p_+ - r, p_-) = - p_- > 0,
\]
which concludes the proof.
\end{proof}

\begin{proof}[Proof of Proposition \ref{prop:rlarge(A)}]
Note that since $F_i'(x) = p_i  - 3 x^2 \leq p_i$ for all $(x,i) \in E$, we have for all $x,y$ and $t \geq 0$, 
\[
| \varphi_t^i(x) - \varphi_t^i(y)| \leq e^{  p_i t} |x - y|,
\]
where $(\varphi^i_t(x))_{t\geq 0, x\in M}$ is the flow of the vector field $F_i$ on $M$. This implies that for all $t \geq t_0$, on
the event $\{t_0 \geq T_0\}$, one has
\begin{equation}
\label{eq:control-after-T0}
d((X_t,I_t),(Y_t,J_t)) \leq e^{  \int_{t_0}^t p_{I_s} ds } d((X_{t_0},I_{t_0}), (Y_{t_0}, J_{t_0})).
\end{equation}
In particular, since when $i = j$, $T_0 = 0$ $\P_{(x,i),(y,j)}$-almost surely, one has using Lemma \ref{lem:r-large-conctrat} that 
\begin{align*}
\E_{(x,i),(y,i)} \left( d((X_t,I_t),(Y_t,J_t)) e^{ - \int_0^t \rho(I_s) ds} \right) & \leq \E_{i}\left( e^{  \int_0^t p_{I_s} ds} Z_t \right) d((x,i),(y,j))\\
& \leq C e^{ - \gamma t} \E_i( Z_t) d((x,i),(y,j)).
\end{align*}
Hence, (A) holds for points $(x,i)$ and $(y,j)$ such that $i=j$. For the case where $i \neq j$, we proceed as in the proof of Proposition \ref{prop:PDMP}. First, let us prove that for some constant $c > 0$, for all $t \geq 0$ and $i \neq j$,
\begin{equation}
\label{eq:initialisation-rec}
\E_{i,j} \left( \mathbbm{1}_{T_0 \leq 1} Z_t^I \right) \geq c \E_i ( Z_t).
\end{equation}
Since the killing acts only on the irreducible Markov chain $I$ living in the finite state space $\mathcal{E}$, (H) is satisfied
(cf.\ e.g.~\cite{ChampagnatVillemonais2016b}): for some $c_H > 0$, for all $t \geq 0$,
\[
\sup_{i \in \mathcal{E}} \E_i Z_t\leq C_H \inf_{j\in\mathcal{E}}\E_j Z_t
\]
Hence, 
\begin{align*}
\E_{i,j} \left( \mathbbm{1}_{T_0 \leq 1} Z_t^I \right)  & = \E_{i,j} \left( \mathbbm{1}_{T_0 \leq 1} Z_1^I \E_{I_1}( Z_{t-1}) \right)\\
& \geq \frac{e^{ - \| \rho \|_{\infty} }}{C_H}  \min_{k,k' \in \mathcal{E}} \P_{k,k'}(T_0 \leq 1)\E_i(Z_{t-1})\\
& \geq \frac{e^{ - osc( \rho)}}{C_H} \min_{k,k' \in \mathcal{E}} \P_{k,k'}(T_0 \leq 1) \E_i(Z_t),
\end{align*}
which proves~\eqref{eq:initialisation-rec} with $c = e^{ - osc( \rho)}  \min_{k,k' \in \mathcal{E}} \P_{k,k'}(T_0 \leq 1) /C_H> 0$.
Reasoning by induction as for the proof of \eqref{eq:recurrence-a.s.PDMP}, this yields for all $k \geq 1$ and all $t \geq k$,
\begin{equation*}
\mathbb E_{(x,i),(y,j)}(\mathbbm{1}_{T_0> k}Z_t^X)\leq (1-c)^k\,\mathbb E_{(x,i)} Z_t.   
\end{equation*}
In particular
\begin{equation}
  \label{eq:recurrence-not-contract}
  \mathbb E_{(x,i),(y,j)}(\mathbbm{1}_{T_0>t/2}Z_t^X)\leq (1-c)^{t/2 -1}\,\mathbb E_{(x,i)} Z_t.   
\end{equation}
Now, by~\eqref{eq:control-after-T0}, we obtain using the Markov property that for all $t \geq 0$,
\begin{multline*}
  \E_{(x,i),(y,j)}\left( Z_t^I \mathbbm{1}_{T_0 \leq t/2} d((X_t, I_t),(Y_t,J_t)) \right) \\
  \begin{aligned}
    & \leq \E_{(x,i),(y,j)}\left( Z_t^I \mathbbm{1}_{T_0 \leq t/2} e^{  \int_{t/2}^t p_{I_s} ds} d\left((X_{t/2},I_{t/2}),(Y_{t/2},J_{t/2})\right)\right)\\
    & \leq \E_{(x,i),(y,j)}\left( Z_{t/2}^I \mathbbm{1}_{T_0 \leq t/2} \E_{I_{t/2}}\left(e^{  \int_0^{t/2} p_{I_s} ds} Z_{t/2}\right) \right).
  \end{aligned}
\end{multline*}
By Lemma \ref{lem:r-large-conctrat} and Markov's property,
\[
  \E_{(x,i),(y,j)}\left( Z_t^I \mathbbm{1}_{T_0 \leq t/2} d((X_t, I_t),(Y_t,J_t)) \right) \leq  C e^{ - \gamma t/2} \E_i Z_t.
\]
Combining this with~\eqref{eq:recurrence-not-contract} completes the proof of (A).
\end{proof}

\subsection{A counter-example when $\rho$ is not Lipschitz}

We now give a counter-example of Theorem~\ref{thm:main-compact} when $p$ is Lipschitz, but  $\rho = - \log(p)$ is not.  We consider the Bernoulli convolution
\[
X_{n+1} = \frac{1}{2} X_n + \theta_{n+1}
\]
with $(\theta_n)_{n\geq 1}$ an i.i.d.\ sequence of Rademacher random variables, on the state space $E = ( - 2, 2) \setminus \{0\}$ with $p(x) = \frac{|x|}{2}$, {i.e. $\rho(x) = - \log(\frac{|x|}{2})$, implying that $\rho$ is not Lipschitz on $E$}.
The next result shows that, even though $X_n$ is contracting almost surely in the sense of \eqref{eq:ascontraction}, the
Wasserstein distance between conditional distributions starting from two different points is bounded below by a positive constant.

\begin{prop}
	For all $x,y \in E$,  for all $n \geq 1$
	\[
	\cW_d \left( \frac{\delta_x P_n}{\delta_x P_n \mathbbm{1}}, \, \frac{\delta_y P_n}{\delta_y P_n \mathbbm{1}} \right) \geq \frac{1}{2} | x - y |.
	\]
\end{prop}


\begin{proof}
	For all $x \in E$, we have 
	\[
	\E_x ( |X_1| ) = \frac{1}{2} \left( \left| \frac{x}{2} + 1 \right| + \left| \frac{x}{2} - 1 \right| \right) =1,
	\]
	and
	\[
	\E_x ( X_1 |X_1|) = \frac{1}{2} \left( \left( \frac{x}{2} + 1 \right)^2 - \left( \frac{x}{2} - 1 \right)^2 \right) = x.
	\]
	Therefore, for all $n \geq 2$,
	\begin{align*}
		\E_x ( |X_1| \cdots |X_{n-1}|) & = \E_x \left( |X_1| \cdots |X_{n-2} | \E_{X_{n-2}}({|X_1|}) \right)\\
		& = \E_x \left( |X_1| \cdots |X_{n-2} | \right) = \ldots = 1,
	\end{align*}
	and thus
		\[
		\E_x( Z_n^X) = \E_x (p(x) \cdots p(X_{n_1})) = \E_x( \frac{|x|}{2} \frac{|X_1|}{2} \cdots \frac{|X_{n-1}|}{2}) =  2^{-n}|x|.
		\]
	On the other hand,
	\begin{align*}
		\E_x ( X_n |X_1| \cdots |X_{n-1}|) &  = \frac{1}{2} \E_x ( X_{n-1} |X_1| \cdots |X_{n-1}|) + \E_x ( \theta_n |X_1| \cdots |X_{n-1}|)\\
		& = \frac{1}{2} \E_x \left( |X_1| \cdots |X_{n-2}| \E_{X_{n-2}} ( X_{n-1} | X_{n-1}|) \right) + 0\\
		& = \frac{1}{2} \E_x \left( |X_1| \cdots | X_{n-2} |X_{n-2} \right) = \ldots = \frac{x}{2},
	\end{align*}
	where we have used independence of $\theta_n$ and $X_1, \ldots, X_{n-1}$ and the fact that $\E ( \theta_n) = 0$. Hence,
		\[
		\E_x( X_n Z_n^X) = 2^{-n} \E_x ( X_n |x| |X_1| \cdots |X_{n-1}|) = 2^{-n} |x| \frac{x}{2},
		\]
		and thus
		\[
		\E_x \left( X_n G_n^X\right) = \frac{x}{2},
		\]
		leading to
		\[
		\left|  \E_x \left( X_n G_n^X\right) -  \E_y \left( Y_n G_n^Y\right) \right| = \frac{|x-y|}{2}.
		\]
	This concludes the proof by Kantorovitch-Rubinstein formula.
\end{proof}

\section*{Acknowledgements}

We thank two anonymous referees for their useful comments.

The work of N.C. is partially funded by the Chair “Modélisation Mathématique et Biodiversité” of VEOLIA-Ecole Polytechnique-MNHN-F.X
and by the European Union (ERC, SINGER, 101054787). Views and opinions expressed are however those of the author(s) only and do not
necessarily reflect those of the European Union or the European Research Council. Neither the European Union nor the granting
authority can be held responsible for them.

\section*{Statements and Declarations}

The authors have no competing interests to declare that are relevant to the content of this article.


\begin{thebibliography}{10}

\bibitem{BansayeEtAl2020}
Vincent Bansaye, Bertrand Cloez, and Pierre Gabriel.
\newblock Ergodic behavior of non-conservative semigroups via generalized
  {D}oeblin’s conditions.
\newblock {\em Acta Applicandae Mathematicae}, 166(1):29--72, 2020.

\bibitem{BardetGuerinEtAl2010}
Jean-Baptiste Bardet, H{\'e}l{\`e}ne Gu{\'e}rin, and Florent Malrieu.
\newblock Long time behavior of diffusions with {M}arkov switching.
\newblock {\em ALEA: Latin American Journal of Probability and Mathematical
  Statistics}, 7:151--170, 2010.

\bibitem{BMZIHP}
M.~Bena{\"\i}m, S.~Le~Borgne, F.~Malrieu, and P.~A. Zitt.
\newblock Qualitative properties of certain piecewise deterministic {M}arkov
  processes.
\newblock {\em Ann. Inst. Henri Poincar{\'e} Probab. Stat.}, 51(3):1040--1075,
  2015.

\bibitem{BenaimChampagnatEtAl2022}
Michel Bena{\"i}m, Nicolas Champagnat, William O{\c c}afrain, and Denis
  Villemonais.
\newblock {Quasi-compactness criterion for strong Feller kernels with an
  application to quasi-stationary distributions}.
\newblock working paper or preprint, April 2022.

\bibitem{BS19}
Michel Bena{\"\i}m and Edouard Strickler.
\newblock Random switching between vector fields having a common zero.
\newblock {\em The Annals of Applied Probability}, 29(1):326--375, 2019.

\bibitem{Birkhoff1957}
Garrett Birkhoff.
\newblock Extensions of {J}entzsch's theorem.
\newblock {\em Trans. Amer. Math. Soc.}, 85:219--227, 1957.

\bibitem{BreyerRoberts1999}
L.~A. Breyer and G.~O. Roberts.
\newblock A quasi-ergodic theorem for evanescent processes.
\newblock {\em Stochastic Process. Appl.}, 84(2):177--186, 1999.

\bibitem{ChampagnatVillemonais2016b}
Nicolas Champagnat and Denis Villemonais.
\newblock Exponential convergence to quasi-stationary distribution and
  {Q}-process.
\newblock {\em Probab. Theory Related Fields}, 164(1):243--283, 2016.

\bibitem{ChampagnatVillemonais2016}
Nicolas Champagnat and Denis Villemonais.
\newblock Uniform convergence to the {Q}-process.
\newblock {\em Electronic Communications in Probability}, 22, 2017.

\bibitem{ChampagnatVillemonais2018b}
Nicolas Champagnat and Denis Villemonais.
\newblock Uniform convergence of penalized time-inhomogeneous {M}arkov
  processes.
\newblock {\em ESAIM: Probability and Statistics}, 22:129--162, 2018.

\bibitem{ChampagnatVillemonais2020}
Nicolas Champagnat and Denis Villemonais.
\newblock Practical criteria for {R}-positive recurrence of unbounded
  semigroups.
\newblock {\em Electronic Communications in Probability}, 25:1--11, 2020.

\bibitem{ChampagnatVillemonais2023}
Nicolas Champagnat and Denis Villemonais.
\newblock {General criteria for the study of quasi-stationarity}.
\newblock {\em Electronic Journal of Probability}, 28(none):1 -- 84, 2023.

\bibitem{CloezThai2016}
Bertrand Cloez and Marie-No\'emie Thai.
\newblock Quantitative results for the {F}leming-{V}iot particle system and
  quasi-stationary distributions in discrete space.
\newblock {\em Stochastic Process. Appl.}, 126(3):680--702, 2016.

\bibitem{ColletMartinezEtAl2013}
Pierre Collet, Servet Mart\'inez, and Jaime San~Mart\'in.
\newblock {\em Quasi-stationary distributions}.
\newblock Probability and its Applications (New York). Springer, Heidelberg,
  2013.
\newblock Markov chains, diffusions and dynamical systems.

\bibitem{DelMoral2004}
Pierre Del~Moral.
\newblock {\em Feynman-{K}ac formulae}.
\newblock Probability and its Applications (New York). Springer-Verlag, New
  York, 2004.
\newblock Genealogical and interacting particle systems with applications.

\bibitem{delMoral2021}
Pierre del Moral and Emma Horton.
\newblock Quantum harmonic oscillators and {F}eynman-{K}ac path integrals for
  linear diffusive particles.
\newblock {\em arXiv preprint arXiv:2106.14592}, 2021.

\bibitem{FerreRoussetEtAl2021}
Gr{\'e}goire Ferr{\'e}, Mathias Rousset, and Gabriel Stoltz.
\newblock More on the long time stability of {F}eynman--{K}ac semigroups.
\newblock {\em Stochastics and Partial Differential Equations: Analysis and
  Computations}, 9:630--673, 2021.

\bibitem{FigalliGigli2010}
Alessio Figalli and Nicola Gigli.
\newblock A new transportation distance between non-negative measures, with
  applications to gradients flows with {D}irichlet boundary conditions.
\newblock {\em Journal de math{\'e}matiques pures et appliqu{\'e}es},
  94(2):107--130, 2010.

\bibitem{GuillinEtAl2020}
Arnaud Guillin, Boris Nectoux, and Liming Wu.
\newblock {Quasi-stationary distribution for strongly Feller Markov processes
  by Lyapunov functions and applications to hypoelliptic Hamiltonian systems}.
\newblock working paper or preprint, December 2020.

\bibitem{HMS11}
Martin Hairer, Jonathan~C Mattingly, and Michael Scheutzow.
\newblock Asymptotic coupling and a general form of harris’ theorem with
  applications to stochastic delay equations.
\newblock {\em Probability theory and related fields}, 149(1-2):223--259, 2011.

\bibitem{HK20}
Tobias Hurth and Christian Kuehn.
\newblock Random switching near bifurcations.
\newblock {\em Stochastics and Dynamics}, 20(02):2050008, 2020.

\bibitem{JournelMonmarche2022}
{Journel, Lucas} and {Monmarch\'e, Pierre}.
\newblock Convergence of a particle approximation for the quasi-stationary
  distribution of a diffusion process: Uniform estimates in a compact soft
  case.
\newblock {\em ESAIM: PS}, 26:1--25, 2022.

\bibitem{KondratyevEtAl2016}
Stanislav Kondratyev, L{\'e}onard Monsaingeon, and Dmitry Vorotnikov.
\newblock {A new optimal transport distance on the space of finite Radon
  measures}.
\newblock {\em Advances in Differential Equations}, 21(11/12):1117 -- 1164,
  2016.

\bibitem{Malrieu2015}
Florent Malrieu.
\newblock Some simple but challenging markov processes.
\newblock In {\em Annales de la Facult{\'e} des sciences de Toulouse:
  Math{\'e}matiques}, volume~24, pages 857--883, 2015.

\bibitem{MeleardVillemonais2012}
Sylvie M{\'e}l{\'e}ard and Denis Villemonais.
\newblock Quasi-stationary distributions and population processes.
\newblock {\em Probab. Surv.}, 9:340--410, 2012.

\bibitem{MeynTweedie2009}
Sean Meyn and Richard~L. Tweedie.
\newblock {\em Markov chains and stochastic stability}.
\newblock Cambridge University Press, Cambridge, second edition, 2009.
\newblock With a prologue by Peter W. Glynn.

\bibitem{Ocafrain2020a}
William O{\c{c}}afrain.
\newblock {Polynomial rate of convergence to the Yaglom limit for Brownian
  motion with drift}.
\newblock {\em Electronic Communications in Probability}, 25(none):1 -- 12,
  2020.

\bibitem{Ocafrain2020}
William O{\c{c}}afrain.
\newblock {Q}-processes and asymptotic properties of {M}arkov processes
  conditioned not to hit moving boundaries.
\newblock {\em Stochastic Processes and their Applications}, 130(6):3445--3476,
  2020.

\bibitem{Ocafrain2021}
William O{\c{c}}afrain.
\newblock Convergence to quasi-stationarity through {P}oincar{\'e} inequalities
  and {B}akry-{E}mery criteria.
\newblock {\em Electronic Journal of Probability}, 26:1--30, 2021.

\bibitem{Peyre2013}
R{\'e}mi Peyre.
\newblock Sharp equivalence between $\rho$- and $\tau$-mixing coefficients.
\newblock {\em Studia Math.}, 216(3):245--270, 2013.

\bibitem{PiccoliRossi2014}
Benedetto Piccoli and Francesco Rossi.
\newblock Generalized {W}asserstein distance and its application to transport
  equations with source.
\newblock {\em Archive for Rational Mechanics and Analysis}, 211:335--358,
  2014.

\bibitem{PiccoliRossi2016}
Benedetto Piccoli and Francesco Rossi.
\newblock On properties of the generalized {W}asserstein distance.
\newblock {\em Arch. Ration. Mech. Anal.}, 222(3):1339--1365, 2016.

\bibitem{DoornPollett2013}
Erik~A. van Doorn and Philip~K. Pollett.
\newblock Quasi-stationary distributions for discrete-state models.
\newblock {\em European J. Oper. Res.}, 230(1):1--14, 2013.

\bibitem{Villani2009}
C\'edric Villani.
\newblock {\em Optimal transport}, volume 338 of {\em Grundlehren der
  Mathematischen Wissenschaften [Fundamental Principles of Mathematical
  Sciences]}.
\newblock Springer-Verlag, Berlin, 2009.
\newblock Old and new.

\bibitem{Villemonais2020}
Denis Villemonais.
\newblock Lower bound for the coarse {R}icci curvature of continuous-time
  pure-jump processes.
\newblock {\em Journal of Theoretical Probability}, 33:954--991, 2020.

\bibitem{Wang2021}
Feng-Yu Wang.
\newblock Precise limit in {W}asserstein distance for conditional empirical
  measures of {D}irichlet diffusion processes.
\newblock {\em Journal of Functional Analysis}, 280(11):108998, 2021.

\end{thebibliography}
\end{document}